\documentclass[reqno, 11pt]{amsart}
\usepackage[top=4cm, bottom=3cm, left=3.2cm, right=3.2cm]{geometry}
\usepackage{enumerate}
\usepackage{amssymb,amsmath}
\usepackage{xcolor}
\usepackage{lineno}
\usepackage{hyperref}
\usepackage{enumitem}

\newtheorem{thm}{Theorem}

\newtheorem{lem}{Lemma}

\newtheorem{assumption}{Assumption}
\theoremstyle{definition}
\newtheorem{defn}{Definition}

\newtheorem*{Claim}{Claim}
\newtheorem*{Example}{Example}
\allowdisplaybreaks

\usepackage{lineno}

\numberwithin{equation}{section} \numberwithin{lem}{section}
\numberwithin{thm}{section} \numberwithin{prop}{section}
\numberwithin{cor}{section} \numberwithin{rem}{section}
\numberwithin{defn}{section}
\numberwithin{assume}{section}

	\newtheorem{remark}{Remark}[section]

	\catcode`@=11 \@addtoreset{equation}{section} \catcode`@=12
\title[Stochastic Keller-Segel system]{  Martingale Solutions to a Stochastic Keller-Segel System with nonlocal Source and Super-linear Noise }

\author{ Qian Li$^{\,1}$, Li Chen$^{\,2}$ and Jinhuan Wang$^{\,1,^*}$ }
\thanks{The work of L.Chen is partially supported by the Deutsche Forschungsgemein-schaft
(DFG, German Research Foundation)-547277619. The work of J. Wang is partially supported by National Natural Science Foundation of China Grants No. 12171218, Liaoning Provincial Natural Science Foundation Program (2024-MS-002).}
\thanks{* Corresponding author: Jinhuan Wang. }
\begin{document}
\maketitle
\begin{center}
{\footnotesize
1-School of Mathematics and Statistics, Liaoning University, Shenyang, 110036, P. R. China \\
Email: liqian990720@163.com; wangjh@lnu.edu.cn \\
 \smallskip
2-School of Business Informatics and Mathematics,
Universit\"{a}t Mannheim, 68131, Mannheim, Germany
\\
Email: li.chen@uni-mannheim.de
}
	\end{center}
\maketitle
 \date{} 
\begin{abstract}
Global nonnegative martingale solutions are shown to exist for a stochastic Keller–Segel system with a nonlocal Fisher–KPP source and super-linear multiplicative noise. The result is obtained for nonnegative initial data with no smallness assumption, provided that the nonlocal source term is dominant. The main difficulty stems from the absence of a coercive structure and the super-linear nature of the noise. An additional cut-off with finite $L^2$ norm in the classical Galerkin method is added to establish a well-posed approximation problem. Moreover, due to the nonlocal Fisher-KPP structure, it is necessary to prove the positivity of the approximating solution in order to obtain uniform estimates. 
In the compactness arguments, the usual tightness argument in the framework of Hilbert spaces (e.g., \cite{law2014existence, dhariwal2019global}) cannot be directly applied to the uniform estimates obtained in this paper. As a result, we develop a more general version of the compactness argument and tightness criterion, presented in the appendix, which will be applied throughout the paper. This allows for the global existence of nonnegative martingale solutions to be derived from Jakubowski's version of the Skorokhod Theorem, along with a thorough discussion of the convergence properties.
\end{abstract}
\maketitle

\maketitle
 \date{}



\section{Introduction}
The Keller-Segel system in chemotaxis modeling was originally introduced by Keller and Segel \cite{keller1970initiation,keller1971model} to describe the directed movement of cells in response to chemical gradients in their environment. The Keller–Segel system with stochastic effects has attracted recent attention, particularly for its biological relevance, including works like \cite{chen2025well, tang2024strong}. In this paper, we study the following stochastic Keller–Segel system with a nonlocal Fisher–KPP source and super-linear multiplicative noise:
\begin{eqnarray}\label{equ1}
\begin{cases}	
	du  = \Delta u~dt-  \nabla  \cdot (u \nabla c) ~dt+ u^\alpha\big(1-\int_\mathcal{O}u^{\beta}dx\big)~dt +\sigma (u)~dW(t),&x\in\mathcal{O},t> 0,
    \\-\Delta c+c=u,&x\in\mathcal{O},t> 0,
    \\\nabla u\cdot \nu=\nabla c\cdot \nu=0,&x\in\partial\mathcal{O},t> 0,
    \\u (x,0) = u_{0}(x)\ge0,&x\in\mathcal{O},
 \end{cases} 
\end{eqnarray}
where $\mathcal{O}\subset\mathbb{R}^d,d\ge 3$ is a bounded domain with a  smooth boundary  and $\nu$ denotes the outward unit normal vector on  $\partial\mathcal{O}$. Here, $u$ represents the bacterial density, $c$ represents the chemical substance concentration,  $u^\alpha(1-\int_\mathcal{O}u^{\beta}dx)$, with $\alpha,\beta>1$, is the Fisher-KPP type reaction term representing non-linear growth under nonlocal resource consumption of the bacteria. Furthermore, $\sigma(u)dW(t)$ represents the driving external random force, where $W(t)$ is $Q$-Wiener process, whose concrete definition is given in Section two.

The deterministic Keller-Segel model has been widely studied in the last decades. While a complete review of this research field is beyond the scope of this paper, it is worth noting that key results include the global existence of solutions, finite-time blow-up phenomena, and the large-time asymptotic behavior of solutions. The classical deterministic model with reaction term can be written as 
\[
u_t  = \Delta u- \nabla  \cdot (u \nabla c) +f(u).
\]
The source term $f(u)$ models cell proliferation and death, where environmental resources may be consumed either locally or nonlocally and it can counteract the blow-up tendency produced by chemotaxis.
A prototypical example is the logistic damping $f(u)=\mu u(1-u)$, $\mu>0$, in which the quadratic absorption term $-\mu u^2$ plays a key role in preventing finite-time blow-up. 
Tello and  Winkler \cite{Tello07062007}  proved that if either $d\le2$ and $\mu>0$ is arbitrary, or $d\ge3$ and $\mu>\frac{d-2}{d}$, then for reasonably regular initial data $u_0$, there exists a global classical solution which remains uniformly bounded for all times.
Motivated by global resource limitation, one may also consider nonlocal logistic-type reactions depending on the total population mass or the total mass of some power of the population, which describes the competition of the individuals of the species for the resources of the environment and the cooperation to  survive. 
Negreanu and Tello  \cite{negreanu2013competitive} considered the logistic-type reaction
$f(u)=u(a_0-a_1u-\frac{a_2}{|\Omega|}\int _\Omega udx)$ with $a_0,a_1>0,a_2\in \mathbb{R}$.
If $a_2>0$, then the nonlocal reaction term contributes a competitive effect that counteracts the aggregation dynamics. Moreover, it is also proved in \cite{negreanu2013competitive} that as the population density increases, the local damping term $a_1u$ becomes dominant compared with the nonlocal term $\frac{a_2}{|\Omega|}\int_{\Omega}u\,dx$, in particular, $f(u)$ behaves asymptotically like $u(a_0-a_1u)$.
To further investigate the influence of the nonlocal non-linear logistic type sources, Bian and Chen \cite{bian2018nonlocal} considered the Fisher-KPP type nonlocal source $f(u)=u^\alpha(1-\int_\Omega u^\beta dx)$ with $\alpha> 1,\beta >1$. They proved global existence of solutions for $\alpha$ and $\beta$ under appropriate conditions.

The Keller-Segel type models for chemotaxis can be derived from stochastic interacting particle systems, as population size tends to infinity. Rigorous derivations of such limits can be found in \cite{stevens2000derivation,stevens2000stochastic}, or   recent works with convergence rate in \cite{bresch2023mean,bresch2019mean,chen2025mean}.  Additionally, the topic of the propagation of chaos and the mean-field limit can be found in \cite{carrillo2014derivation,jabin2017mean}.
For the deterministic Keller-Segel model, the underlying particle system is typically driven only by mutually independent noises.  In more realistic settings, however, particle systems may also be influenced by environmental noises, such as temperature, light, or sound, which do not vanish in the large population limit, these are intrinsic to a more realistic setting such as culturing bacteria. A natural way to account for these effects is to incorporate multiplicative noise into the deterministic dynamics, leading to stochastic Keller-Segel type models. Related rigorous derivations can be found in \cite{chen2025hegselmann, huang2021microscopic,nikolaev2025quantitative}. On the macroscopic level, the stochastic Keller-Segel system has been intensively studied in the last couple of years. Among them, the well-posedness and qualitative properties of the following stochastic Keller-Segel type system have been established.
\[
du = [\Delta u- \nabla  \cdot (u \nabla c)]dt +\sigma(u)dW(t).
\]
Flandoli, Galeati and Luo \cite{flandoli2021delayed} proved local existence for a class of non-linear SPDEs, including stochastic Keller-Segel equations on the torus. 
Huang and Qiu \cite{huang2021microscopic} proved  the noise is of divergence type, and  obtained local existence and uniqueness of weak solutions on an interval $[0,T]$, where the existence time $T$ depends on the $L^4$-norm of the initial data. They further showed that if the $L^4$-norm of the initial data is sufficiently small, then a global weak solution exists. 
Misiats, Stanzhytskyi and Topaloglu \cite{misiats2022global} investigated how different stochastic perturbations affect the solution behavior, in particular  considering divergence-form noise $\sigma(u)=\nabla u$  and linear multiplicative noise $\sigma(u)=u$. For divergence form noise, they obtained global weak solutions for small initial data and  for linear multiplicative noise, they proved finite time blow-up with positive probability for any non-trivial initial data.
Tang and Wang \cite{tang2024strong} studied  multiplicative noise and established local existence and pathwise uniqueness of strong solutions on $\mathbb R^d$ ($d\ge2$). For non-autonomous linear noise, they further proved global existence and large-time decay properties of strong solutions. When the noise is non-linear and sufficiently strong, they showed that it can prevent blow-up. Hausenblas, Mukherjee and Tran \cite{hausenblas2022one}  proved the existence of a martingale solution for a one-dimensional stochastic Keller-Segel system, where the noise is interpreted in the Stratonovich sense. 

If additional (local) logistic source or non-linear noise is added, Chen, Zhai and Zhang \cite{chen2025well} obtained the well-posedness of the stochastic Keller-Segel system on a two-dimensional bounded domain. More precisely, they proved the existence of a unique global nonnegative mild solution for any nonnegative initial data under either of assumptions: the noise $\sigma(u)$ is linear and the logistic type source $f(u)$ satisfies $f(0)\ge 0,~f(s)\leq c_1-\mu s^2$ where $c_1,\mu >0$ or the noise $\sigma(u)=\|u\|_q^ru$ is non-linear and the logistic type source $f(u)$ satisfies $f(0)\ge 0,~f(s)\leq c_2+\mu' |s|^n$.  From the viewpoint of stochastic analysis, in our work, we introduce in \eqref{equ1} a superlinear multiplicative
noise to capture state-dependent perturbations whose intensity increases more than linearly with the
population density. To counterbalance this stochastic amplification, we incorporate a nonlocal
Fisher-KPP term, which induces an effective global damping when $\int_{\mathcal O} u^\beta\,dx$ becomes
large and is essential for obtaining uniform estimates and compactness. The stochastic
Keller-Segel system with genuinely superlinear multiplicative noise without such a nonlocal
damping mechanism remains largely open, which will be investigated in  future projects.  This requires definitely new ideas will likely  to establish global well-posedness.

In this work, we study the stochastic Keller-Segel  nonlocal Fisher-KPP system \eqref{equ1} with multiplicative super-linear noise and establish the global existence of nonnegative martingale solutions by using Galerkin approximation and tightness argument.
The existence results for martingale solutions and the corresponding methods to other types of stochastic partial differential equations can be found in, e.g., 
\cite{braukhoff2024global,brzezniak2013existence,brzeźniak2019fractionally,brzezniak2017invariant,law2014existence,dhariwal2019global,tang2018pathwise}.

For the problem we considered in this paper, the main difficulty arises from the absence of coercivity in the Galerkin scheme and the super-linear noise in the stochastic integral. To overcome these difficulties, we introduce an additional cut-off in the Galerkin method, ensuring a well-posed approximation. In the uniform estimate, the aggregation effect and the super-linear noise are both dominated by the nonlocal source term, applying delicate Sobolev inequalities, which is the most technical step in the proof. Furthermore, since the standard tightness argument in Hilbert spaces (e.g., \cite{law2014existence, dhariwal2019global}) does not apply here, we develop a more general compactness argument and tightness criterion, provided in the appendix. This enables the global existence of nonnegative martingale solutions via Jakubowski’s Skorokhod Theorem, with a detailed convergence analysis.

\vskip5mm
Next, we briefly outline the main proof ideas, novelties, and arrangement of this paper.  

In Section two, we introduce the basic setting for the $Q$-Wiener process and state the assumptions on the noise coefficient $\sigma(u)$. We give the definitions of solutions used in this work and present the main result. 

In Section three, we formulate a stochastic Galerkin approximation equation (\ref{equ2}) of the system (\ref{equ1}) on finite-dimensional space $H_n$  and show that the equation (\ref{equ2}) admits a nonnegative strong solution $u_n$.  
     This problem can not be solved directly by classical theory for stochastic differential systems due to the lack of the  weak coercivity condition. Instead, we prove this result in two steps.
     \begin{itemize}
     	\item \underline{Step 1.} We introduce a smooth cut-off function for $\|u_n\|_{L^2(\mathcal{O})}\leq N$ and the positive part $u_{n}^{+}$ to ensure that the reaction term  $(u_{n}^{+})^\alpha(1-\int_{\mathcal O}(u_{n}^{+})^\beta dx)$ is well-defined. This makes it possible to obtain that for each $N>0$, the truncated equation admits a pathwise unique global nonnegative strong solution $u_n^N\in C([0,\infty);H_n)$ $\mathbb{P}$-a.s.  Therefore, we conclude that the Galerkin equation \eqref{equ2} admits a nonnegative local solution $u_n^N(\cdot\wedge\tau_n^N)\in C([0,\infty);H_n)$ $\mathbb{P}$-a.s., where the stopping time $\tau_n^N$ comes from the cut-off $\|u_n\|_{L^2(\mathcal{O})}\leq N$.
        \item \underline{Step 2.} We establish pathwise uniqueness for the local strong solutions and thus obtain the existence of the maximal nonnegative local solution $(u_n,\tau^*)$  to equation \eqref{equ2}. Furthermore, under the given assumptions for $\alpha$, $\beta$, and $\gamma$,  we show that the nonlocal Fisher-KPP source  term can be used to prevent finite-time blow-up arising from the chemotactic term and the super-linear noise term.  Consequently, uniform estimates for $u_n$ can be obtained by using Sobolev type inequalities and It\^{o}'s formula. This is the key technical ingredient in the analysis of this paper.  Thus, we can derive $\tau^*=T$ almost surely. Therefore, \eqref{equ2} admits a pathwise unique global nonnegative solution   $u_n$.
          \end{itemize}

In Section four, we construct a global martingale solution $ ((\tilde{\Omega} ,\tilde{\mathcal{F}}, \{\tilde{\mathcal{F}}_{t}\} _{t\ge 0}  ,\tilde{\mathbb{P}}  ), \tilde{u},\tilde{W})$  to the equation (\ref{equ1}), by using a general tightness criterion. 
\begin{itemize}
\item \underline{Step 1.} Based on the uniform estimates obtained for $u_n$ from \eqref{equ2}, we proceed the compactness argument to finalize the proof. In the context of the compactness arguments, the usual tightness argument, which is commonly applied in the framework of Hilbert spaces (as seen in works like \cite{law2014existence, dhariwal2019global}), cannot be directly used in this case due to the uniform estimates obtained for $u_n$. Actually, it allows us to use the more general  topological space 
\vskip-7mm
$$
\hspace*{2cm}\mathcal{Z}
:= C([0,T];\mathbb U') \cap L^{q}_{w}(0,T;\mathbb{V}) \cap L^{p}(0,T;E) \cap C([0,T];\mathbb{H}_{w}),\,1<q\le p<\infty,
$$ 
where $\mathbb U$, $\mathbb V$, and $E$ are not necessary Hilbert spaces.
To address this, we develop a more general version of the compactness argument, which incorporates a revised tightness criterion. This is provided in detail in the appendix. The criterion is more flexible and accounts for the specific structure of the stochastic Keller-Segel system under consideration. 
Then we obtain  tightness of the laws $\{\mathcal{L}\big((u_n,W_n)\big):n\in\mathbb{N}\}$. 
\item \underline{Step 2.} Applying   Jakubowski's version of  the Skorokhod Theorem for non-metric spaces, there exists a subsequence $\{(u_n,W_n):n\in\mathbb{N}\}$, which is not relabeled, a probability space $(\tilde{\Omega} ,\tilde{\mathcal{F}}, \{\tilde{\mathcal{F}}_{t}\} _{t\ge 0}  ,\tilde{\mathbb{P}}  )$ and random variables $(\tilde{u},\tilde{W})$, $(\tilde{u}_n,\tilde{W}_n)$ with $n\in\mathbb{N}$ such that $(\tilde{u}_n,\tilde{W}_n)$ has the same law as $({u}_n,{W}_n)$. Furthermore, $
        (\tilde{u}_n,\tilde{W}_n)\to (\tilde{u},\tilde{W})$, $\tilde{\mathbb{P}}$-a.s. Finally, using the uniform estimates for $\tilde u_n$ together with Vitali's convergence theorem, we can prove the convergences in the non-linear terms and conclude that
 $(  (\tilde{\Omega} ,\tilde{\mathcal{F}}, \{\tilde{\mathcal{F}}_{t}\} _{t\ge 0}  ,\tilde{\mathbb{P}}  ), \tilde{u},\tilde{W})$ is a global martingale solution to the equation (\ref{equ1}).
\end{itemize}

 \section{Preliminaries and the main result}
In this section, we set up the functional analytic and probabilistic framework throughout the paper. We introduce the notation for the relevant function spaces and recall basic facts on the $Q$-Wiener process and the associated stochastic integral;  see \cite[Chapter 2]{liu2015stochastic} for details. We give the assumptions on  the noise coefficient $\sigma(u)$ for this paper. We provide  some auxiliary results which will be used in the analysis and  uniform estimates. Additionally, we present the definition of Martingale solution and the main result. 

 \subsection{Q-Wiener process}  
  Let $(U,\langle,\rangle_U)$  be a separable Hilbert space. The space of all bounded linear operators from $U$ to $L^2(\mathcal{O})$ is denoted by $L(U,L^2(\mathcal{O}))$, particularly we write $L(U)$ instead of $L(U,U)$. Let $\xi_l,l\in\mathbb{N}$ be an orthonormal basis of $U$, the space of Hilbert-Schmidt operators from $U$ to $L^2(\mathcal{O})$ is defined by 
 \begin{align*}
     L_2(U,L^2(\mathcal{O}))=\Big\{B\in L(U,L^2(\mathcal{O})):\sum _{l=1}^{\infty}\|B\xi_l\|_{L^2(\mathcal{O})}^2<\infty\Big\}
 \end{align*}
 and it is endowed with the norm $\|B\|_{L_2(U,{L^2(\mathcal{O})})}^2=\sum _{l=1}^{\infty}\|B\xi_l\|_{L^2(\mathcal{O})}^2$. 
 
 Let $Q\in L(U)$ be a nonnegative, symmetric and  finite trace operator, defined by 
 \[
 Q\xi_l=\lambda_l\xi_l,\quad\lambda_l\ge 0,l\in\mathbb{N},\quad\sum_{l}^\infty\lambda_l<\infty.
 \]
 Let   $(\Omega,\mathcal F,\{\mathcal F_t\}_{t\ge0},\mathbb P)$ be a probability space  with a complete, right-continuous filtration. We denote $\{W_l(t)\}_{l\in\mathbb N}$ to be independent real-valued $\{\mathcal F_t\}_{t\ge 0}$-Brownian motions. The $U$-valued $Q$-Wiener process is given by
	\[
	W^Q(t):=\sum_{l=1}^\infty \sqrt{\lambda_l}\,\xi_l\,W_l(t),\qquad t\in[0,T].
	\]
	Denote $U_0:=Q^{1/2}(U)$ equipped with the inner product
	\[
	\langle f_1,f_2\rangle_{U_0}:=\langle Q^{-1/2}f_1,Q^{-1/2}f_2\rangle_{U},\quad\forall f_1,\,f_2\in U_0.
	\]
	Then $Q^{1/2}\xi_l=\sqrt{\lambda_l}\xi_l$ and $\{\sqrt{\lambda_l}\xi_l\}_{l\in\mathbb N}$ is an orthonormal basis of $U_0$. We denote the space of stochastically integrable processes by
    \begin{align*}
        &\mathcal N_W^2(0,T;{L^2(\mathcal{O})})\\
	:=&\Big\{\Phi:\Omega\times[0,T]\to L_2(U_0,{L^2(\mathcal{O})})\ \Big|\ 
	\Phi\ \text{is predictable}, \,\mathbb E\int_0^T\|\Phi(s)\|_{L_2(U_0,{L^2(\mathcal{O})})}^2\,ds<\infty
	\Big\}.
    \end{align*}
	For any $\Phi\in\mathcal N_W^2(0,T;{L^2(\mathcal{O})})$ and $t\in[0,T]$, the It\^o integral $\displaystyle\int_0^t \Phi(s)\,dW^Q(s)$ is well defined and given by
	\[
	\int_0^t \Phi(s)\,dW^Q(s)
	=\sum_{l=1}^\infty \int_0^t \Phi(s)\big(Q^{1/2}\xi_l\big)\,dW_l(s)
	=\sum_{l=1}^\infty \int_0^t \Phi(s)\big(\sqrt{\lambda_l}\xi_l\big)\,dW_l(s),
	\quad \mathbb P\text{-a.s.},
	\]
	where the series converges in $L^2(\Omega;C([0,T];{L^2(\mathcal{O})}))$. Moreover, the Burkholder-Davis-Gundy inequality yields (see, e.g., \cite{gawarecki2010stochastic}, Lemma 3.1)
	\begin{equation}\label{BDG}
		\mathbb E\Big[\sup_{t\in[0,T]}\Big\|\int_0^t \Phi(s)\,dW^Q(s)\Big\|_{L^2(\mathcal{O})}^{2p}\Big]
		\le C_p\,\mathbb E\Big(\int_0^T \|\Phi(s)\|_{L_2(U_0,{L^2(\mathcal{O})})}^2\,ds\Big)^p,\quad\forall p\ge1,
	\end{equation}
	for some constant $C_p>0$ depending only on $p$. 
	
 \subsection{Assumptions on the noise coefficient.}
 Let $\{e_l\}_{l\in\mathbb N}$ be an orthonormal basis of $L^2(\mathcal O)$ consisting of eigenfunctions of the Neumann Laplacian, i.e.
	\begin{align}\label{f9}
		-\Delta_N e_l=\mu_l e_l \quad\text{in }\mathcal O,\quad
		\nabla e_l\cdot\nu=0 \quad\text{on }\partial\mathcal O,
	\end{align}
    where $\nu$ is the outward unit normal vector on $\partial\mathcal{O}$ and $0=\mu_1\le \mu_2\le\cdots$ and $\mu_l\to\infty$ as $l\to\infty$.
Let us impose the following assumptions on the noise coefficient  $\sigma$. This super-linear multiplicative noise is motivated by \cite{salins2022global}.
    \begin{assumption}\label{ass}
    Assume that $\sigma: L^\gamma(\mathcal O)\to L_2\big(U_0,L^2(\mathcal O)\big)$.
        For all $u,v\in L^\gamma(\mathcal{O}) $,  $\gamma\ge 2$,  there exists a constant $C>0$ such  that 
\begin{align}
  \label{ass3}   &\|\sigma(u)\|_{L_2{(U_0,L^2(\mathcal{O}))}}^2\leq C\|u\|_{L^\gamma(\mathcal{O})}^\gamma;\\
   \label{ass4} &\|\sigma(u)-\sigma(v)\|_{L_2{(U_0,L^2(\mathcal{O}))}}^2\leq C(\|u\|_{L^\gamma(\mathcal{O})}^{\gamma-2}+\|v\|_{L^\gamma(\mathcal{O})}^{\gamma-2})\|u-v\|_{L^\gamma(\mathcal{O})}^2.
\end{align} 	
In particular, for every measurable set $B\subset\mathcal O$ and all $u,v\in L^\gamma(\mathcal O)$ such that 
		$u=v$ a.e. on $B$, it holds
		\begin{equation}\label{ass2}
			\mathbf 1_B\,\sigma(u)(Q^\frac{1}{2}\xi_l)=\mathbf 1_B\,\sigma(v)(Q^\frac{1}{2}\xi_l),\quad\forall \,l\in\mathbb N.
		\end{equation}
    \end{assumption}
Consequently, if $u$ is an $L^\gamma(\mathcal O)$-valued predictable process  with
	$\mathbb E\displaystyle\int_0^T\|u(s)\|_{L^\gamma(\mathcal{O})}^\gamma ds<\infty$, 	then the stochastic integral $\displaystyle\int_0^t \sigma(u(s))\,dW^Q(s)$ is well defined and satisfies
	\[
	\int_0^t \sigma(u(s))\,dW^Q(s)
	=\sum_{l=1}^\infty \int_0^t \sigma(u(s))\big(Q^{1/2}\xi_l\big)\,dW_l(s),
	\qquad \mathbb P\text{-a.s.},
	\]
	where the series converges in $L^2(\Omega;C([0,T];L^2(\mathcal O)))$. In particular, it
is a continuous $L^2(\mathcal O)$-valued martingale.

For convenience, we use the following notation for the stochastic integral in this paper.
$$
\int_0^t \sigma(u(s))\,dW(s):=\int_0^t \sigma(u(s))\,dW^Q(s).
$$

\begin{remark}
    	 These conditions   play a crucial role in constructing the nonnegativity of solutions. The condition \eqref{ass3} implies that $\sigma(u)=0$ if $u=0$.  Together with \eqref{ass2},  it follows that   for every $u\in L^\gamma(\mathcal O)$ and each  $l\in\mathbb{N}$,
         \begin{align}\label{ass5}
            \mathbf 1_{\{u<0\}}\sigma(u)(Q^\frac{1}{2}\xi_l)=\sigma(u\wedge 0)(Q^\frac{1}{2}\xi_l).
         \end{align}
         In fact, let $B:=\{x\in\mathcal O:\ u(x)<0\}$. Since $u=u\wedge0$ a.e. on $B$, by \eqref{ass2} we have
	\[
	\mathbf 1_B\sigma(u)(Q^\frac{1}{2}\xi_l)=\mathbf 1_B\sigma(u\wedge0)(Q^\frac{1}{2}\xi_l).
	\]
	Moreover, since $u\wedge0=0$ a.e. on $B^c$ and  $\sigma(0)=0$, using \eqref{ass2} yields
	\[
	\mathbf 1_{B^c}\sigma(u\wedge0)(Q^\frac{1}{2}\xi_l)=\mathbf 1_{B^c}\sigma(0)(Q^\frac{1}{2}\xi_l)=0.
	\]
	 Consequence,
	\[
	\sigma(u\wedge0)(Q^\frac{1}{2}\xi_l)=\mathbf 1_B\sigma(u\wedge 0)(Q^\frac{1}{2}\xi_l)=\mathbf 1_B\sigma(u)(Q^\frac{1}{2}\xi_l)=\mathbf 1_{\{u<0\}}\sigma(u)(Q^\frac{1}{2}\xi_l).
	\]
\end{remark}
\vskip3mm

Next, we give a concrete example to show that the Assumption \ref{ass} for the nonlinear noise is reasonable. 
\begin{Example}
    In this example, we take $U:=L^2(\mathcal{O}),\,U_0:=Q^{1/2}(L^2(\mathcal{O}))$. 
	   Now we set  $Q=(I-\Delta_N)^{-\rho},\,$ namely
\begin{align}
    Qe_l=\lambda_le_l,~\lambda_l=(1+\mu_l)^{-\rho},
\end{align}
where $\{e_l\}_{l\in\mathbb{N}}$ is an orthonormal basis of $L^2(\mathcal{O})$ established in (\ref{f9}).
By  the fact  $\mu_l\sim l^\frac{2}{d}$  and $\|e_l\|_{L^\infty(\mathcal{O})}\leq Cl^\frac{d-1}{4}$ (\cite{grieser2002uniform}, Theorem 1),  we know that for   $\rho>d-\frac{1}{2}$, it holds 
\begin{align}
    \sum _{l=1}^{\infty}\lambda_l\|e_l\|_{L^\infty(\mathcal{O})}^2< \infty.
\end{align} 
We define $\sigma(u):U_0\to  L^2(\mathcal{O})$ by
\begin{align}\label{f8}
    \sigma (u)(Q^\frac{1}{2}h)=\sum_{l=1}^{\infty}\sqrt{\lambda_l}\langle e_l,h\rangle_{L^2(\mathcal O)}e_l\cdot g(u),\quad \forall h\in L^2(\mathcal O),
\end{align}
where $g:\mathbb{R}\to \mathbb{R}$ and $g(r)= a|r|^\frac{\gamma-2}{2}r$ with a positive constant  $a$.

 Next, we verify that the noise coefficient $\sigma$
defined above satisfies all the conditions in the Assumption \ref{ass}. By the definition of  the Hilbert-Schmidt norm, we  get
\begin{align*}
\|\sigma(u)\|_{L_2{(U_0,L^2(\mathcal{O}))}}^2&=\sum _{l=1}^{\infty}{\lambda_l}\int_\mathcal{O}\big| e_l(x)g(u(x))\big|^2dx\\
   &\leq \sum _{l=1}^{\infty}\lambda_l\|e_l\|_{L^\infty(\mathcal{O})}^2\int_\mathcal{O}|g(u(x))|^2dx\\&\leq a^2\sum_{l=1}^{\infty}\lambda_l\|e_l\|_{L^\infty(\mathcal{O})}^2\int_\mathcal{O}|u(x)|^\gamma dx\leq C\|u\|_{L^\gamma(\mathcal{O})}^\gamma.
\end{align*}
This finishes the proof of \eqref{ass3}.
 In order to prove \eqref{ass4}, we introduce the following  elementary inequality, for any $r,s\in \mathbb{R},p\ge2$
\begin{align}\label{d11}
  \big||r|^{p-2}r-|s|^{p-2}s\big|\leq  ({p-1})(|r|^{p-2}+|s|^{p-2})|r-s|.
\end{align}
In our example,  by the above inequality and the H\"older inequality, we have 
\begin{align*}
    &\|\sigma(u)-\sigma(v)\|_{L_2{(U_0,L^2(\mathcal{O}))}}^2\\=&\sum _{l=1}^{\infty}{\lambda_l}\int_\mathcal{O}\big|e_l(x)g(u(x))-e_l(x)g(v(x))\big|^2dx\\\leq& \sum_{l=1}^{\infty}\lambda_l\|e_l\|_{L^\infty(\mathcal{O})}^2\int_\mathcal{O}\big|g(u(x))-g(v(x))\big|^2dx\\\leq &(\frac{a\gamma}{2})^2\sum_{l=1}^{\infty}\lambda_l\|e_l\|_{L^\infty(\mathcal{O})}^2\int_\mathcal{O}(|u(x)|^{\gamma-2}+|v(x)|^{\gamma-2})|u(x)-v(x)|^2dx\\\leq& C(\|u\|_{L^\gamma(\mathcal{O})}^{\gamma-2}+\|v\|_{L^\gamma(\mathcal{O})}^{\gamma-2})\|u-v\|_{L^\gamma(\mathcal{O})}^2.
\end{align*}
	Then, we verify \eqref{ass2}. Let $B\subset\mathcal O$ be measurable and assume $u=v$ a.e.\ on $B$.
	Then $g(u)=g(v)$ a.e.\ on $B$, and for every $l$,
	\[
	\mathbf 1_B\,\sigma(u)(Q^{1/2}e_l)
	=\mathbf 1_B\,\sqrt{\lambda_l}e_l\,g(u)
	=\mathbf 1_B\,\sqrt{\lambda_l}e_l\,g(v)
	=\mathbf 1_B\,\sigma(v)(Q^{1/2}e_l).
	\]
Therefore,  this example holds true for \eqref{ass4}.
\end{Example}
\subsection{The  elliptic  problem.}
  Let $c$ be a solution satisfying
 \begin{eqnarray}\label{equ3}
\begin{cases}	
	-\Delta c+c=u,&x\in\mathcal{O},
    \\\nabla c\cdot \nu=0,&x\in\partial\mathcal{O}.
 \end{cases} 
\end{eqnarray}
If $u\in L^p(\mathcal {O})$ with $1<p<\infty$, applying the classical elliptic regularity  theory  (\cite{grisvard2011elliptic}, Chapter 2) and (\cite{gilbarg1977elliptic}, Chapter 6.7, Chapter 9.5),   then there exists a unique  solution $c\in W^{2,p}(\mathcal{O})$, moreover
\begin{align}\label{c}
    \|c\|_{W^{2,p}(\mathcal{O})}\leq C\|u\|_{L^p(\mathcal{O})},
\end{align}
where $C$ depends on $d,p,\mathcal{O}$. By the Sobolev embedding $W^{2,p}(\mathcal{O})\hookrightarrow C^{1,1-\frac{d}{p}}(\mathcal{O}),\,p>d$, it holds 
\begin{align}\label{d}
    \|\nabla c\|_{L^\infty(\mathcal{O})}+\|c\|_{L^\infty(\mathcal{O})}\leq C\|u\|_{L^p(\mathcal{O})}.
\end{align}
Actually, in the above classical elliptic regularity theory $\partial\mathcal{O}\in C^{1,1}$ is needed.

\subsection{Gagliardo-Nirenberg inequality}
The following lemma given in \cite{bian2016nonlocal} will be used frequently in this paper to  get the uniform estimates.
\begin{lem}\label{lem1}
    Let $p = \frac{2d}{d-2}$, $1 \leq r < q < p$ and $\frac{q}{r} < \frac{2}{r} + 1 - \frac{2}{p}$, then for $v\in H^1(\mathcal{O}) \cap L^r(\mathcal{O})$, it holds
\[
\|v\|_{L^q(\mathcal{O})}^q \leq C(d)\left( C_0^{-\frac{\lambda q}{2 - \lambda q}} + C_1^{-\frac{\lambda q}{2 - \lambda q}} \right) \|v\|_{L^r(\mathcal{O})}^\zeta + C_0 \|\nabla v\|_{L^2(\mathcal{O})}^2 + C_1 \|v\|_{L^2(\mathcal{O})}^2, \quad d \geq 3.
\]
Here, $C$ is a constant relying on $d$, and $C_0, C_1 > 0$ are arbitrary constants,
\[
\lambda = \frac{\frac{1}{r} - \frac{1}{q}}{\frac{1}{r} - \frac{1}{p}}\in(0,1), \quad \text{and} \quad \zeta = \frac{2-\frac{2q}{p}}{\frac{2 - q}{r} - \frac{2}{p} + 1}.
\]
  \end{lem}
\subsection{Finite Dimensional Projection}
	Let $\{e_l\}_{l\in\mathbb{N}}$ be an orthonormal basis of $L^2(\mathcal O)$ given in (\ref{f9}).
Fix $n\in \mathbb{N}$, let $H_n:=\text{span}\{e_1,...,e_n\}$ with the norm inherited from $L^2(\mathcal{O})$. Let $\Pi_n$ be the projection operator from $L^2(\mathcal{O})$ to $H_n$ defined by 
\[
\Pi_nv=\sum_{l=1}^{n} (v,e_l)_{L^2(\mathcal{O})}e_l,\quad \forall v\in L^2(\mathcal{O}).
\]
Then  $\Pi_n$ is self-adjoint and contractive on $L^2(\mathcal O)$
i.e.,
\begin{align}
	&(\Pi_n u, v)_{L^2(\mathcal O)} = (u,\Pi_n v)_{L^2(\mathcal O)}, \qquad \forall u,v\in L^2(\mathcal O), \label{3.1}\\
	&\|\Pi_n u\|_{L^2(\mathcal O)} \le \|u\|_{L^2(\mathcal O)}, \qquad \forall u\in L^2(\mathcal O). \label{3.2}
\end{align}

\subsection{Main Theorem}
For convenience, we first give the definition of Martingale solution to the stochastic Keller-Segel problem \eqref{equ1}.
  \begin{defn}(Martingale solution)\label{def}
  Let $T > 0$, a global martingale solution on $[0,T]$ to the equation (\ref{equ1}) is a triple\begin{align*}
 \big(  (\tilde{\Omega} ,\tilde{\mathcal{F}}, \{\tilde{\mathcal{F}}_{t}\} _{t\ge 0}  ,\tilde{\mathbb{P}}  ), \tilde{u},\tilde{W}\big) 
  \end{align*}such that
  \item[(i)]$\tilde{\mathfrak{A}}:= (\tilde{\Omega} ,\tilde{\mathcal{F}}, \{\tilde{\mathcal{F}}_{t}\} _{t\ge 0}  ,\tilde{\mathbb{P}}  )$ is a probability space with a complete, right-continuous filtration;
   \item[(ii)]$\tilde{W}$ is a $Q$ Wiener process   on $\tilde{\mathfrak{A}}$;
 \item[(iii)]$\tilde{u}(t)$ is a $\tilde{\mathcal{F}}_{t}$-measurable process  such that
\begin{align*}
   \tilde{u } \in L^2(\tilde{\Omega};L^\infty([0,T];L^2(\mathcal{O})))\cap L^2(\tilde{\Omega};L^2([0,T];H^1(\mathcal{O})))\cap L^k(\tilde{\Omega};L^k([0,T];L^k(\mathcal{O}))),
\end{align*}
where $ k\in(2,\frac{2}{d+2}(\alpha-1+\beta)+2)$. Moreover,
the law of $\tilde{u}(0)$ is the same as  $u_0$ and  for all $t\in[0,T]$ and for all $\phi \in H^1(\mathcal{O})$,  $\tilde{u}$ satisfies
\begin{align*}
    (\tilde{u}(t),\phi)_{L^2(\mathcal{O})}=&(\tilde{u}(0),\phi)_{L^2(\mathcal{O})}-\int _0^t\int_\mathcal{O}\nabla\tilde{u}(s)\cdot\nabla\phi dx ds+  \int _0^t\int_\mathcal{O} \tilde{u}(s)\nabla \tilde{c}(s)\cdot\nabla  \phi dxds
    \\&+\int _0^t\int_\mathcal{O} \tilde{u}^\alpha(s)(1-\int_\mathcal{O}\tilde{u}^{\beta}(s)dx)\phi dx ds+\Big(\int _0^t\sigma(\tilde{u}(s))d\tilde{W}(s),\phi \Big)_{L^2(\mathcal{O})}.
\end{align*}
 \end{defn}

 \begin{thm}\label{thm1}(Global existence)
 Assume that there exists a nonnegative sequence $u_n(0)\in L^2(\Omega,\mathcal{F}_0;H_n)$ which satisfy that  $\sup_{n\ge1}\mathbb{E}\|u_n(0)\|_{L^2(\mathcal{O})}^2<\infty$  such that $u_n(0)\to u_0$ in $L^2(\Omega;L^2(\mathcal{O}))$.   Let the spatial dimension $d\ge 3$,  the constant $\alpha>1,\beta>1,\gamma\ge2$ and $\alpha+1>\beta$. Under Assumption \ref{ass}, if \[\max\{3,\gamma,\alpha+1\}<\frac{2}{d+2}(\alpha-1+\beta),\]  then for every $T>0$,  there exists a global nonnegative martingale solution \[\big(  (\tilde{\Omega} ,\tilde{\mathcal{F}}, \{\tilde{\mathcal{F}}_{t}\} _{t\ge 0}  ,\tilde{\mathbb{P}}  ), \tilde{u},\tilde{W}\big )\] on $[0,T]$
to the equation (\ref{equ1}) in the sense  of Definition \ref{def}. 
\end{thm}
Actually, for a given initial data  $u_0\in L^2(\Omega,\mathcal{F}_0;L^2(\mathcal{O}))$ and $u_0\ge0 \,a.e.$ in $\mathcal{O}$, it has been mentioned in \cite{zhai20202d} that such approximation $u_n(0)$ in Theorem \ref{thm1} exists.
 \section{Well-posedness and uniform estimates of Galerkin approximate system } 

In this section, we first study the local well-posedness of the following classical Galerkin approximation in the space $H_n$,
\begin{eqnarray}\label{equ2}
	\begin{cases}	
		du_n  = \Pi_n\Big[\Delta u_n-  \nabla  \cdot (u_n \nabla c_n)+ u_n^\alpha(1-\int_\mathcal{O}u_n^{\beta}dx)\Big]~dt +\Pi_n\sigma (u_n)~dW(t),
		\\-\Delta c_n+c_n=u_n,\quad \text{in }\mathcal{O}
		\\\nabla u_n\cdot\nu =\nabla c_n\cdot\nu=0, \quad \text{on }\partial\mathcal{O}
	\end{cases} 
\end{eqnarray}
with initial $u_n(0)\in H_n$. 
Then, we show that this approximation sequence is uniformly bounded in certain norms, consequently we obtain the global well-posedness of \eqref{equ2}. At the same time, we obtain that these global bounds are preparations for the compactness argument in the next section. 

\subsection{Local Strong Solution}

We first prove that \eqref{equ2} has a unique local strong solution. For convenience, we repeat the definition of it.
\begin{defn}\label{def1}(Local strong  solution and uniqueness) Fix a  probability space ${\mathfrak{A}}:= ({\Omega} ,{\mathcal{F}}$, $ \{{\mathcal{F}}_{t}\} _{t\ge 0},  {\mathbb{P}}  )$ with a complete, right-continuous filtration and let $W(t)$ be a $Q$-Wiener process on ${\mathfrak{A}}$.  Let $u_0\in L^2(\Omega,\mathcal{F}_0;L^2(\mathcal{O}))$ and $u_0\ge 0\,a.e.$ in $\mathcal O$. A local strong solution to (\ref{equ2}) is  $u_n^N(\cdot \wedge{\tau_n^N})$, where $\tau_n^N$ is a stopping time  and $u_n^N:\Omega\times[0,\infty)\to H_n$ is an  $H_n$-valued  $\{\mathcal{F}_t\}_{t\ge 0}$-predictable process satisfying 
	\[
	u_n^N(\cdot\wedge\tau_n^N)\in C([0,\infty);H_n)~\mathbb{P}\text{-a.s.},
	\]
	and for all $t\ge0$,
	\begin{align*}
		u_n^N(t\wedge\tau_n^N )=&u_n(0)+\int_0^{t\wedge\tau_n^N}\Pi_n\big[\Delta u_n^N-  \nabla  \cdot (u_n^N \nabla c_n^N) +(u_n^N)^\alpha(1-\int_\mathcal{O}(u_n^{N})^{\beta}dx)\big]~dt' \\&+\int_0^{t\wedge\tau_n^N}\Pi_n\sigma (u_n^N)~dW(t')\quad\mathbb{P}\text{-a.s.}
	\end{align*} 
	The local solution is pathwise unique, if any two local solution $u_n^{N_1}(\cdot \wedge{\tau_n^{N_1}})$ and $u_n^{N_2}(\cdot \wedge{\tau_n^{N_2}})$ satisfy that $u_n^{N_1}(0)=u_n^{N_2}(0)$~$\mathbb{P}$-a.s. can imply $u_n^{N_1}(t)=u_n^{N_2}(t)$~$\mathbb{P}$-a.s. on $t\in[0,\tau_n^{N_1}\wedge\tau_n^{N_2})$.
\end{defn}

We prove the existence and uniqueness of local strong solution separately in the following two theorem.

\begin{thm}\label{thm3.1}
    Let $d\ge 3$ and  $\alpha> 1,\beta>1,\gamma\ge2$.   Under  Assumption \ref{ass}, for each $n\in\mathbb{N}$ and each $N>0$, the Galerkin equation (\ref{equ2}) exists a pathwise unique nonnegative local strong solution $u_n^N(\cdot\wedge{{\tau_n^N}})$  in the sense  of Definition \ref{def1}, where the stopping time ${{\tau_n^N}}:=\inf\{t\ge0:\|u_n^N(t)\|_{L^2(\mathcal{O})}\ge N\}$. Moreover, $u_n^N(\cdot\wedge{\tau_n^N})\in C([0,\infty);H_n),\,\mathbb{P}$-a.s. 
\end{thm}
\begin{proof}
Let $\theta:[0,\infty)\to[0,1]$ be a $C^2$ function such that $\theta(r)=1$ for $r\in [0,1]$, $\theta(r)=0$ for $r\ge2$ and $\sup_{r\in[0,\infty)}|\theta'(r)|\leq C_\theta<\infty$. For any $N>0$, let $\theta_N(r)=\theta(\frac{r}{N})$. Let $r^{+}=\max\{0,r\}$ denote the positive part. For each $u_n\in H_n$ , let $c_{n}(u_n)$ be a unique solution of the second equation (\ref{equ2}). Since $H_n$ is finite dimensional, for 
 $p\in[1,\infty]$, $L^p(\mathcal{O})$ norms are equivalent on $H_n$, there exists $C(n,p)$ such that 
 \begin{align}\label{n2}
     \|v\|_{L^p(\mathcal{O})}\leq C(n,p)\|v\|_{L^2(\mathcal{O})},\quad \forall v\in H_n.
 \end{align}
 Moreover, by (\ref{d}), (\ref{n2}), we obtain that for $p>d$,
 \begin{align}\label{n3}
     \|\nabla c_n(v)\|_{L^\infty(\mathcal{O})}+\|c_n(v)\|_{L^\infty(\mathcal{O})}\leq \|v\|_{L^p(\mathcal{O})}\leq C(n,p)\|v\|_{L^2(\mathcal{O})},\quad \forall v\in H_n.
 \end{align}
 Next, we consider the following truncated  system of \eqref{equ2} 
 \begin{align}\label{equ4}
     du_{n}^{N}
			=a(u_n^N)dt+b(u_n^N)dW(t),
 \end{align}
where the function $a:H_n\to H_n$ is defined by 
\begin{align*}
     a(u_n^N)= \Pi_n\Big[&\Delta u_{n}^{N}- \theta_N(\|u_{n}^{N}\|_{L^2(\mathcal{O})})\nabla\cdot\big(u_{n}^{N}\nabla c_n(u_{n}^{N})\big)\\
			&+\theta_N(\|u_{n}^{N}\|_{L^2(\mathcal{O})})(u_{n}^{N+})^\alpha\Big(1-\int_{\mathcal O}(u_{n}^{N+})^\beta dx\Big)
			\Big],
\end{align*}
 the function $b:H_n\to L_2(U_0,H_n)$ is defined by \[b(u_n^N)=\Pi_n\Big[\theta_N(\|u_{n}^{N}\|_{L^2(\mathcal{O})})\sigma(u_{n}^{N})\Big].\] 
 In fact, if $\theta_N(\|u_n^N\|_{L^2(\mathcal{O})})\neq0$, then $\|u_n^N\|_{L^2(\mathcal{O})}\le 2N$. 
 If $\theta_N(\|u_n^N\|_{L^2(\mathcal{O})})=0$, the chemotactic  term, the logistic term and the noise term vanish. Then, we will use the well-known theory of stochastic differential equations (\cite{liu2015stochastic}, Theorem 3.1.1) to prove the existence of the global solution on the truncated equation \eqref{equ4}.
 
  Fix $N>0$,
 let  $u_n^N,v_n^N\in H_n$ and $\|u_n^N\|_{L^2(\mathcal{O})},\|v_n^N\|_{L^2(\mathcal{O})}\leq 2N$. We prove $a,\,b$ satisfy the  local weak monotonicity condition. By \eqref{3.1}, we obtain  
 \begin{align*}
 &\big(a(u_n^N)-a(v_n^N),{u_n^N}-{v_n^N}\big)_{H_n}\\=&-\int_\mathcal{O}\nabla ({u_n^N}-{v_n^N})\cdot\nabla({u_n^N}-{v_n^N})dx\\&+ \int_\mathcal{O}\big(\theta_N(\|{u_n^N}\|_{L^2(\mathcal{O})}){u_n^N}\nabla c_{n}({u_n^N})-\theta_N(\|{v_n^N}\|_{L^2(\mathcal{O})}){v_n^N}\nabla c_{n}({v_n^N})\big)\cdot\nabla({u_n^N}-{v_n^N})dx
 \\&+\int_\mathcal{O}\big(\theta_N(\|{u_n^N}\|_{L^2(\mathcal{O})}) ({u_n^{N+}})^\alpha(1-\int_\mathcal{O}({u_n^{N+}})^{\beta}dx)\big)({u_n^N}-{v_n^N})dx\\&- \int_\mathcal{O}\big(\theta_N(\|{v_n^N}\|_{L^2(\mathcal{O})})({v_n^{N+}})^\alpha(1-\int_\mathcal{O}({v_n^{N+}})^{\beta}dx)\big)({u_n^N}-{v_n^N})dx\\
 \leq&-\int_\mathcal{O}|\nabla ({u_n^N}-{v_n^N})|^2dx+  \Big|\int_\mathcal{O}\theta_N(\|{u_n^N}\|_{L^2(\mathcal{O})})({u_n^N}-{v_n^N})\nabla c_{n}({u_n^N})\cdot\nabla({u_n^N}-{v_n^N})dx\Big|\\
 &+  \Big|\int_\mathcal{O}\theta_N(\|{u_n^N}\|_{L^2(\mathcal{O})}){v_n^N}\big(\nabla c_{n}({u_n^N})-\nabla c_{n}({v_n^N})\big)\cdot\nabla({u_n^N}-{v_n^N})dx\Big|\\&+ \Big|\int_\mathcal{O}\big(\theta_N(\|{u_n^N}\|_{L^2(\mathcal{O})})-\theta_N(\|{v_n^N}\|_{L^2(\mathcal{O})})\big){v_n^N}\nabla c_{n}({v_n^N})\cdot\nabla({u_n^N}-{v_n^N})dx\Big|\\&+\Big|\int_\mathcal{O}\big(\theta_N(\|{u_n^N}\|_{L^2(\mathcal{O})})({u_n^{N+}})^\alpha-\theta_N(\|{v_n^N}\|_{L^2(\mathcal{O})})({v_n^{N+}})^\alpha\big)(1-\int_\mathcal{O}({u_n^{N+}})^\beta dx)({u_n^N}-{v_n^N})dx\Big|\\&+\Big|\int_\mathcal{O}\big(({u_n^{N+}})^\beta-({v_n^{N+}})^\beta\big)dx\int_\mathcal{O}\theta_N(\|{v_n^N}\|_{L^2(\mathcal{O})})({v_n^{N+}})^\alpha({u_n^N}-{v_n^N})dx\Big|\\=&:-\|\nabla ({u_n^N}-{v_n^N})\|_{L^2(\mathcal{O})}^2+\sum_{i=1}^{5}I_i.
 \end{align*}
 	For $I_1$, by the H\"older inequality,  \eqref{n3} and Young's inequality,
	\begin{align*}
		|I_1|
		&\le  \,\|{u_n^N}-{v_n^N}\|_{L^2(\mathcal{O})}\,\|\nabla c_{n}(u_n^N)\|_{L^\infty(\mathcal{O})}\,\|\nabla({u_n^N}-{v_n^N})\|_{L^2(\mathcal{O})}\\
		&\le \frac14\|\nabla({u_n^N}-{v_n^N})\|_{L^2(\mathcal{O})}^2 + C( n,N)\|{u_n^N}-{v_n^N}\|_{L^2(\mathcal{O})}^2.
	\end{align*}
		Similarly, using  \eqref{n3}
	\begin{align*}
		| I_2|
		&\le  \,\|{v_n^N}\|_{L^2(\mathcal{O})}\,\|\nabla\big(c_{n}({u_n^N})-c_{n}({v_n^N})\big)\|_{L^\infty(\mathcal{O})}\|\nabla({u_n^N}-{v_n^N})\|_{L^2(\mathcal{O})}\\
		&\le \frac14\|\nabla({u_n^N}-{v_n^N})\|_{L^2(\mathcal{O})}^2 + C( n,N)\|{u_n^N}-{v_n^N}\|_{L^2(\mathcal{O})}^2.
	\end{align*}
	For $I_3$, since $\theta_N$ is Lipschitz and
	$\big|\|{u_n^N}\|_{L^2(\mathcal{O})}-\|{v_n^N}\|_{L^2(\mathcal{O})}\big|\le \|{u_n^N}-{v_n^N}\|_{L^2(\mathcal{O})}$, we have
	\begin{align}\label{n7}
	    \Big|\theta_N(\|{u_n^N}\|_{L^2(\mathcal{O})})-\theta_N(\|{v_n^N}\|_{L^2(\mathcal{O})})\Big|
	\le \frac{C_\theta}{N}\|{u_n^N}-{v_n^N}\|_{L^2(\mathcal{O})}.
	\end{align}
	Hence, using \eqref{n2}, \eqref{n3} and Young's inequality yields
	\begin{align*}
		|I_3|
		&\le C(  n,N)\|{u_n^N}-{v_n^N}\|_{L^2(\mathcal{O})}\,\|{v_n^N}\|_{L^2(\mathcal{O})}\,\|\nabla c_{n}({v_n^N})\|_{L^\infty(\mathcal{O})}\,\|\nabla({u_n^N}-{v_n^N})\|_{L^2(\mathcal{O})}\\
		&\le \frac14\|\nabla({u_n^N}-{v_n^N})\|_{L^2(\mathcal{O})}^2 + C( n,N)\|{u_n^N}-{v_n^N}\|_{L^2(\mathcal{O})}^2.
	\end{align*}
    By the fact that the maps
	$s\mapsto s^a$ is locally Lipschitz, it holds
    \[ |s_1^a-s_2^a|\leq a(s_1^{a-1}+s_2^{a-1})|s_1-s_2|,\quad \forall s_1,s_2\ge0,a>1.\]
    Since $\|{u_n^{N+}}\|_{L^\infty(\mathcal{O})}+\|{v_n^{N+}}\|_{L^\infty(\mathcal{O})}\le C(n,N)$, we have 
    \begin{align}\label{n5}
        \|({u_n^{N+}})^\alpha-({v_n^{N+}})^\alpha\|_{L^2(\mathcal{O})}\le C(\alpha,n,N)\|{u_n^{N+}}-{v_n^{N+}}\|_{L^2(\mathcal{O})},\\\label{n6}
	\|({u_n^{N+}})^\beta-({v_n^{N+}})^\beta\|_{L^2(\mathcal{O})}\le C(\beta,n,N)\|{u_n^{N+}}-{v_n^{N+}}\|_{L^2(\mathcal{O})}.
    \end{align}
    For $I_4,$ by \eqref{n2},  we use the Lipschitz property of $\theta_N(\|\cdot\|_{L^2(\mathcal {O})})$ and the above estimates \eqref{n5} to infer that 
\begin{align*}
    |I_4|\leq&\big|\theta_N(\|{u_n^N}\|_{L^2(\mathcal{O})})\big|\|({u_n^{N+}})^\alpha-({v_n^{N+}})^\alpha\|_{L^2(\mathcal{O})}(1+\|{u_n^{N+}}\|_{L^\beta(\mathcal{O})}^\beta)\|{u_n^{N}}-{v_n^{N}}\|_{L^2(\mathcal{O})} \\&+\big|\theta_N(\|{u_n^N}\|_{L^2(\mathcal{O})})-\theta_N(\|{v_n^N}\|_{L^2(\mathcal{O})})\big|\|{v_n^{N+}}\|_{L^{2\alpha}(\mathcal{O})}^\alpha(1+\|{u_n^{N+}}\|_{L^\beta(\mathcal{O})}^\beta)\|{u_n^{N}}-{v_n^{N}}\|_{L^2(\mathcal{O})} 
    \\\leq& C(\alpha,\beta,n,N)\|{u_n^{N}}-{v_n^{N}}\|_{L^2(\mathcal{O})}^2.
\end{align*}
For $I_5$,  by the embedding and \eqref{n6}, we know
\begin{align*}
    |I_5|\leq&\big|\theta_N(\|{u_n^N}\|_{L^2(\mathcal{O})})\big|\|({u_n^{N+}})^\beta-({v_n^{N+}})^\beta\|_{L^1(\mathcal{O})}\|{u_n^{N+}}\|_{L^{2\alpha}(\mathcal{O})}^\alpha\|{u_n^{N}}-{v_n^{N}}\|_{L^2(\mathcal{O})}\\\leq&C(\alpha,n,N)\big|\|({u_n^{N+}})^\beta-({v_n^{N+}})^\beta\|_{L^2(\mathcal{O})}\|{u_n^{N}}-{v_n^{N}}\|_{L^2(\mathcal{O})}\\\leq &C(\alpha,\beta,n,N)\|{u_n^{N}}-{v_n^{N}}\|_{L^2(\mathcal{O})}^2.
\end{align*}
 Hence, combining the bounds for $I_i,\,1\leq i\leq 5$, we get
	\[
	\big(a({u_n^N})-a({v_n^N}),{u_n^N}-{v_n^N}\big)_{H_n}
	\le  C( \alpha,\beta,n,N)\|{u_n^N}-{v_n^N}\|_{H_n}^2.
	\]
	Together with Assumption  \ref{ass} on $\sigma$, by (\ref{3.2}), (\ref{n2}) and \eqref{n7}, we have 
    \begin{align*}
      &\|b({u_n^N})-b({v_n^N})\|_{L_2(U_0,H_n)}^2\\\leq
      & \|\theta_N(\|{u_n^N}\|_{L^2(\mathcal{O})})\sigma({u_n^N})-\theta_N(\|{v_n^N}\|_{L^2(\mathcal{O})})\sigma({v_n^N})\|_{L_2(U_0,L^2(\mathcal{O}))}^2\\\le&2\big|\theta_N(\|{u_n^N}\|_{L^2(\mathcal{O})})-\theta_N(\|{v_n^N}\|_{L^2(\mathcal{O})})\big|^2\|\sigma({v_n^N})\|_{L_2(U_0,L^2(\mathcal{O}))}^2\\&+2\big|\theta_N(\|{u_n^N}\|_{L^2(\mathcal{O})})\big|^2\|\sigma({u_n^N})-\sigma({v_n^N})\|_{L_2(U_0,L^2(\mathcal{O}))}^2\\\le &C(n,\gamma,N)\|{u_n^N}-v_n^N\|_{L^2(\mathcal{O})}^2=C(n,\gamma,N)\|{u_n^N}-v_n^N\|_{H_n}^2.  
    \end{align*}
	Therefore, combining with the above two estimates, the local weak monotonicity condition holds.
 Next,  we verify the weak coercivity condition  for $a,\,b$.  By  (\ref{ass3}), (\ref{n2}) and the definition of $\theta_N$, we have 
 \begin{align*}
    & \big(a({u_n^N}),{u_n^N}\big)_{H_n}+\|b({u_n^N})\|_{L_2{(U_0,H_n)}}^2\\
    \le&-\int_\mathcal{O}|\nabla {u_n^N}|^2dx+  \int_\mathcal{O}\theta_N(\|{u_n^N}\|_{L^2(\mathcal{O})}){u_n^N} \nabla c_n(u_n^N)\cdot\nabla {u_n^N}dx\\&+\int_\mathcal{O}\theta_N(\|{u_n^N}\|_{L^2(\mathcal{O})})|{u_n^{N}}|^{\alpha+1}(1+\int_\mathcal{O}|{u_n^{N}}|^{\beta}dx)dx+|\theta_N(\|{u_n^N}\|_{L^2(\mathcal{O})})|^2\|\sigma ({u_n^N})\|_{L_2{(U_0,L^2(\mathcal{O}))}}^2\\
    \leq &-\|\nabla {u_n^N}\|_{L^2(\mathcal{O})}^2+ \|{u_n^N}\|_{L^2(\mathcal{O})}\|\nabla c_n(u_n^N)\|_{L^\infty(\mathcal{O})}\|\nabla {u_n^N}\|_{L^2(\mathcal{O})}\\&+\|{u_n^{N}}\|_{L^{\alpha+1}(\mathcal{O})}^{\alpha+1}+\|{u_n^N}\|_{L^{\alpha+1}(\mathcal{O})}^{\alpha+1}\|{u_n^N}\|_{L^\beta(\mathcal{O})}^\beta+C\|{u_n^N}\|_{L^\gamma(\mathcal{O})}^\gamma\\
    \leq& -\frac{1}{2}\|\nabla {u_n^N}\|_{L^2(\mathcal{O})}^2+C( n,N)\|{u_n^N}\|_{L^2(\mathcal{O})}^2+C(\alpha
    ,\beta,\gamma,n,N)\\
    \leq& C(\alpha,\beta,\gamma,n,N)(1+\|{u_n^N}\|_{L^2(\mathcal{O})}^2)=C(\alpha,\beta,\gamma,n,N)(1+\|{u_n^N}\|_{H_n}^2).
 \end{align*}
 This verifies the  weak coercivity condition. 	Therefore, by the classical finite-dimensional SDE theory for the truncated equation (\ref{equ4}), 
  for each $N>0$, there exists a unique global strong solution $u_n^N$ which is $\{\mathcal{F}_t\}_{t\ge 0}$-adapted process  such that
	\[
	u_{n}^{N}\in C([0,\infty);H_n)\quad \mathbb P\text{-a.s.}
	\]
Next, we show that if $u_n(0)\ge 0$ a.e. in $\mathcal{O}$, then  $\mathbb{P}( u_n^N(t)\ge 0\,a.e.\text{in}\,\mathcal{O},\,\forall t\in[0,\infty) )=1$. 
The proof of nonnegative  is analogous (\cite{tang2024strong}, Theorem 2.1). Let \begin{align*}
           F(u_n^N):=\Pi_n\Big[&- \theta_N(\|u_{n}^{N}\|_{L^2(\mathcal{O})})\nabla\cdot\big(u_{n}^{N}\nabla c_n(u_{n}^{N})\big)\\
			&+\theta_N(\|u_{n}^{N}\|_{L^2(\mathcal{O})})(u_{n}^{N+})^\alpha\Big(1-\int_{\mathcal O}(u_{n}^{N+})^\beta dx\Big)
			\Big].
       \end{align*}
       Observe that $\theta_N(\|u_n^N\|_{L^2(\mathcal{O})})\neq0$ implies $\|u_n^N\|_{L^2(\mathcal{O})}\le 2N$ and 
$\theta_N(\|u_n^N\|_{L^2(\mathcal{O})})=0$ implies $F(u_n^N)=0$ and $b(u_n^N)=0$. Let $\rho(u_n^N):=-\mathbf{1}_{\{{u_n^N}<0\}}{u_n^N}$. Similar to the estimate of $\| \rho(u_n^N(t))\|_{L^2(\mathcal O)}^2$, namely applying   It\^o's formula  to a regularization of $\rho(u_n^N(t))$ and then passing to the limit,
(see \cite{tang2024strong}), we obtain
\begin{align*}
	&\mathbb{E}\big(\|\rho({u_n^N}(t))\|_{L^2(\mathcal{O})}^2\big)-\mathbb{E}\big(\|\rho({u_n}(0))\|_{L^2(\mathcal{O})}^2\big)\\
	=&-2\mathbb E\Big(\int_0^t\|\nabla(-\rho(u_n^N(s)))\|_{L^2(\mathcal{O})}^2ds\Big)
	+2\mathbb E\Big(\int_0^t(\rho(u_n^N(s)),F(u_n^N(s)))_{L^2(\mathcal{O})}ds\Big) \\
	&+\mathbb{E}\Big(\int_0^{t}\|\sigma(-\rho({u_n^N}(s)))\|_{L_2(U_0,{L^2(\mathcal{O})})}^2ds\Big),
\end{align*}
    where we use the fact that $\nabla(\mathbf{1}_{\{{u_n^N}<0\}}{u_n^N})=\nabla {u_n^N}$ on $\{{u_n^N}<0\}$ and  by Assumption \ref{ass} and  \eqref{ass5}, we know $\mathbf{1}_{\{{u_n^N}<0\}}\sigma(u_n^N)=\sigma(-\rho(u_n^N)).$ 
     Since $u_n^N(t)\in H_n$, by the properties of orthogonal projection operators of \eqref{3.2} and the estimate \eqref{n2}, we infer that 
       \begin{align*}
           \nonumber\|F(u_n^N)\|_{L^2(\mathcal{O})}&\leq \| \nabla  \cdot (u_n^N \nabla c_n(u_n^N))+(u_n^{N+})^\alpha(1-\int_\mathcal{O}(u_n^{N+})^{\beta}dx)\|_{L^2(\mathcal{O})}\\\nonumber&\leq \|\nabla u_n^N\nabla c_n(u_n^N)+u_n^N\Delta
           c_n(u_n^N)\|_{L^2(\mathcal{O})}+\|u_n^{N+}\|_{L^{2\alpha}(\mathcal{O})}^\alpha(1+\|u_n^{N+}\|_{L^\beta(\mathcal{O})}^\beta)\\&\leq C(  n,N)\|\nabla u_n^N\|_{L^2(\mathcal{O})}+C( n,N)\| u_n^N\|_{L^2(\mathcal{O})}+C(\alpha,\beta,n,N)\| u_n^N\|_{L^2(\mathcal{O})}.
       \end{align*}
          Consequently, we have 
       \begin{align*}
         &\big |\big(\rho({u_n^N}),F({u_n^N})\big)_{L^2(\mathcal{O})}\big|\\\leq& C( n,N) \|\rho({u_n^N})\|_{L^2(\mathcal{O})}\|\nabla(-\rho({u_n^N}))\|_{L^2(\mathcal{O})}+C( \alpha,\beta,n,N) \|\rho({u_n^N})\|_{L^2(\mathcal{O})}^2\\ \leq & C( \alpha,\beta,n,N) \|\rho({u_n^N})\|_{L^2(\mathcal{O})}^2+\|\nabla(-\rho({u_n^N}))\|_{L^2(\mathcal{O})}^2.
       \end{align*}
       
       Therefore, by Assumption \ref{ass}, we have 
       \begin{align*}
           &\mathbb{E}\big(\|\rho({u_n^N}(t))\|_{L^2(\mathcal{O})}^2\big)-\mathbb{E}\big(\|\rho({u_n}(0))\|_{L^2(\mathcal{O})}^2\big)\\=&-\mathbb{E}\Big(\int_0^{t}\|\nabla(-\rho({u_n^N}(s)))\|_{L^2(\mathcal{O})}^2ds\Big)+\mathbb{E}\Big(\int_0^{t}\|\sigma(-\rho({u_n^N}(s)))\|_{L_2(U_0,{L^2(\mathcal{O})}}^2ds\Big)\\&+C( \alpha,\beta,n,N) \mathbb{E}\Big(\int_0^{t}\|\rho({u_n^N}(s))\|_{L^2(\mathcal{O})}^2ds\Big)\\\leq  &C( \alpha,\beta,n,N)\mathbb{E}\Big(\int_0^{t}\|\rho({u_n^N}(s))\|_{L^2(\mathcal{O})}^2 ds\Big).
       \end{align*}
       Since $\rho({u_n^N}(t))=0\, a.e.$ in $\mathcal{O}$, using the Gronwall lemma, for any $t\in[0,\infty
       )$, we get \begin{align*}
      \mathbb{E}\big(\|\rho({u_n^N}(t))\|_{L^2(\mathcal{O})}^2\big)=0.
       \end{align*}
       Hence $\rho({u_n^N}(t))=0$ in $\mathcal O$, $\mathbb P$-a.s., for each $t\in[0,\infty)$. Since $u_n^N$ is path continuous,  it implies $\mathbb{P}( u_n^N(t)\ge 0\,a.e.\text{in}\,\mathcal{O},\,\forall t\in[0,\infty) )=1$.
       This completes the proof of nonnegativity.
       
Define the stopping time
\[
\tau_n^{N}:=\inf\{t\ge0:\|u_n^N(t)\|_{L^2(\mathcal{O})}\ge N\}.
\]
  By the definition of $\theta_N$,
we have $\theta_N(\|u_n^N(t)\|_{L^2(\mathcal {O})})\equiv1$ on 
$[0,\tau_n^{N}]$. Moreover, by the  nonnegativity of $u_n^N$, for $t\in[0,\tau_n^N]$, we have 
\[(u_n^{N+}(t))^\alpha=(u_n^{N}(t))^\alpha,\quad(u_n^{N+}(t))^\beta=(u_n^{N}(t))^\beta, \,\mathbb{P}\text{-a.s.}\]
Consequently, for any  $t\in[0,\infty)$, the equation (\ref{equ2}) exists a the unique nonnegative local strong solution  $u_n^N(t\wedge\tau_{N})$ which is  $\{\mathcal{F}_t\}_{t\ge 0}$-adapted process and  $u_n^N(\cdot\wedge\tau_n^N)\in C([0,\infty);H_n),\,\mathbb{P}$-a.s.
This finishes the proof of Theorem \ref{thm3.1}.
\end{proof}
Next, we prove the pathwise uniqueness of  local strong solution.
\begin{thm}\label{thm3.2}
       Under  Assumption \ref{ass}, let  $u_n^{N_1}(\cdot\wedge\tau_n^{N_1})$ and $u_n^{N_2}(\cdot\wedge\tau_n^{N_2})$ be two local strong solutions to (\ref{equ2}) in the sense of Definition \ref{def1}  defined on the same stochastic basis. If  $\mathbb{P}\{u_n^{N_1}(0)=u_n^{N_2}(0)\}=1$ and let $\tau=\tau_n^{N_1}\wedge\tau_n^{N_2}$, then  we have that for all $t\in[0,\tau)$,
       \[
       u_n^{N_1}(t)=u_n^{N_2}(t),\quad\mathbb{P}\text{-a.s.}
       \]
   \end{thm}
   \begin{proof}
       We define \[
       \tau_M^1=\inf\{t\ge 0:\|u_n^{N_1}(t)\|_{L^2(\mathcal{O})}\ge M\}\wedge\tau_n^{N_1},~\tau_M^2=\inf\{t\ge 0:\|u_n^{N_2}(t)\|_{L^2(\mathcal{O})}\ge M\}\wedge\tau_n^{N_2}.
       \]
       Let $\tau_M=\tau_M^1\wedge\tau_M^2$ and $w_n=u_n^{N_1}-u_n^{N_2}$. By the definition of local solution, for $t\in[0,T\wedge\tau_M]$,  we have that 
\begin{align*}
    w_n(t)=&{u_n^{N_1}}(0)-{u_n^{N_2}}(0)+\int_0^t\Pi_n\Delta w_n(s)ds+\int_0^t\Pi_n\big(\sigma({u_n^{N_1}}(s))-\sigma({u_n^{N_2}}(s))\big)dW(s)\\&- \int_0^t\Pi_n\nabla\cdot\big(w_n(s)\nabla c_{n}(u_n^{N_1})(s)+{u_n^{N_2}}(s)\nabla\big(c_{n}(u_n^{N_1}(s))-c_{n}(u_n^{N_2}(s))\big)\big)ds\\&+\int_0^t\big(1-\int _\mathcal{O}({u_n^{N_1}}(s))^\beta dx\big)\Pi_n\big(({u_n^{N_1}}(s))^\alpha-({u_n^{N_2}}(s))^\alpha \big)ds\\&+\int_0^t\big(\int _\mathcal{O}({u_n^{N_2}}(s))^\beta dx-\int _\mathcal{O}({u_n^{N_1}}(s))^\beta dx\big)\Pi_n({u_n^{N_2}}(s))^\alpha ds.
\end{align*}
Applying It$\hat{\text{o}}$ formula (\cite{liu2015stochastic}, Theorem 4.2.5) to $\|w_n(t)\|_{L^2(\mathcal{O})}^2$ for $t\in[0,T\wedge\tau_M]$, by (\ref{3.1}), we get 
\begin{align*}
    &\|w_n(t)\|_{L^2(\mathcal{O})}^2-\|{u_n^{N_1}}(0)-{u_n^{N_2}}(0)\|_{L^2(\mathcal{O})}^2\\=&2\int_0^t\big(w_n(s),\Delta w_n(s)\big)_{L^2(\mathcal{O})}ds+2\int_0^t\big(w_n(s),\Pi_n\big(\sigma({u_n^{N_1}}(s))-\sigma({u_n^{N_2}}(s))\big)dW(s)\big)_{L^2(\mathcal{O})}\\&+2 \int_0^t\big(\nabla w_n(s),w_n(s)\nabla c_{n}({u_n^{N_1}})(s)+{u_n^{N_2}}(s)\nabla(c_{n}({u_n^{N_1}}(s))-c_{n}({u_n^{N_2}}(s)))\big)_{L^2(\mathcal{O})}ds\\&+2\int_0^t(1-\int _\mathcal{O}({u_n^{N_1}}(s))^\beta dx)\big(w_n(s),({u_n^{N_1}}(s))^\alpha-({u_n^{N_2}}(s))^\alpha\big)_{L^2(\mathcal{O})}ds\\&+2\int_0^t\big(\int _\mathcal{O}({u_n^{N_2}}(s))^\beta dx-\int _\mathcal{O}({u_n^{N_1}}(s))^\beta dx\big)\big(w_n(s),({u_n^{N_2}}(s))^\alpha\big)_{L^2(\mathcal{O})}ds\\&+\int_0^t\|\Pi_n\big(\sigma({u_n^{N_1}}(s))-\sigma({u_n^{N_2}}(s))\big)\|_{L_2(U_0,L^2(\mathcal{O}))}^2ds\\
    \leq&-2\int_0^t\|\nabla w_n(s)\|_{L^2(\mathcal{O})}^2ds+2\int_0^t\big(w_n(s),\Pi_n\big(\sigma({u_n^{N_1}}(s))-\sigma({u_n^{N_2}}(s))\big)dW(s)\big)_{L^2(\mathcal{O})}\\
    &+2 \int_0^t\|\nabla w_n(s)\|_{L^2(\mathcal{O})}\| w_n(s)\|_{L^2(\mathcal{O})}\|\nabla c_{n}({u_n^{N_1}}(s))\|_{L^\infty(\mathcal{O})}ds\\&+2 \int_0^t\|\nabla w_n(s)\|_{L^2(\mathcal{O})}\|{u_n^{N_2}}(s)\|_{L^2(\mathcal{O})}\|\nabla\big(c_{n}({u_n^{N_1}}(s))-c_{n}({u_n^{N_2}}(s))\big)\|_{L^{\infty}(\mathcal{O})}ds\\&+2\int_0^t(1+\|{u_n^{N_1}}(s)\|_{L^\beta(\mathcal {O})}^\beta)\|w_n(s)\|_{L^2(\mathcal{O})}\|({u_n^{N_1}}(s))^\alpha-({u_n^{N_2}}(s))^\alpha\|_{L^2(\mathcal{O})}ds\\&+2\int_0^t\|({u_n^{N_1}}(s))^\beta-({u_n^{N_2}}(s))^\beta \|_{L^1(\mathcal{O})}\|({u_n^{N_2}}(s))^\alpha\|_{L^2(\mathcal {O})}\|w_n(s)\|_{L^2(\mathcal{O})}ds\\&+\int_0^t\|\sigma({u_n^{N_1}}(s))-\sigma({u_n^{N_2}}(s))\|_{L_2(U_0,L^2(\mathcal{O}))}^2ds.
\end{align*}
Since $\|{u_n^{N_1}}(t)\|_{L^2(\mathcal{O})},\|{u_n^{N_2}}(t)\|_{L^2(\mathcal{O})}<M$, for $t\in[0,T\wedge\tau_M]$, then we use the similar ways with the proof of local weak monotonicity condition to obtain
\begin{align*}
    &\|w_n(t)\|_{L^2(\mathcal{O})}^2-\|{u_n^{N_1}}(0)-{u_n^{N_2}}(0)\|_{L^2(\mathcal{O})}^2\\\leq&2\Big|\int_0^t\big(w_n(s),\Pi_n\big(\sigma({u_n^{N_1}})-\sigma({u_n^{N_2}})\big)dW(s)\big)_{L^2(\mathcal{O})}\Big|+C( \alpha,\beta,\gamma,n,M)\int_0^t\|w_n(s)\|_{L^2(\mathcal{O})}^2ds.
\end{align*}
      Taking a supremum over $t\in[0,T\wedge\tau_M]$ and the expectation,
      using the
Burkholder-Davis-Gundy inequality, by (\ref{ass4}) and Young's inequality, we get
\begin{align}
    \nonumber &2\mathbb{E}\Big(\sup_{t\in[0,T\wedge\tau_M]}\Big|\int_{0}^{t}\big(w_n(s), \Pi_n\big(\sigma ({u_n^{N_1}}(s))-\sigma({u_n^{N_2}}(s))\big)dW(s)\big)_{L^2(\mathcal{O})}\Big|\Big)\\\nonumber
\leq&C\mathbb{E}\Big(\int_{0}^{T\wedge\tau_M}\|w_n(t)\|_{L^2(\mathcal{O})}^2\|\sigma ({u_n^{N_1}}(t))-\sigma ({u_n^{N_2}}(t))\|_{L_2(U_0,L^2(\mathcal{O}))}^2dt\Big)^\frac{1}{2}\\\nonumber
\leq&C\mathbb{E}\Big(\Big(\sup_{t\in[0,T\wedge\tau_M]}\|w_n(t)\|_{L^2(\mathcal{O})}^2\Big)^\frac{1}{2}\Big(\int_{0}^{T\wedge\tau_M}\|\sigma ({u_n^{N_1}}(t))-\sigma ({u_n^{N_2}}(t))\|_{L_2(U_0,L^2(\mathcal{O}))}^2dt\Big)^\frac{1}{2}\Big)\\
    \leq&\frac{1}{2}\mathbb{E}\Big(\sup_{t\in[0,T\wedge\tau_M]}\|w_n(t)\|_{L^2(\mathcal{O})}^2\Big)+C(\gamma,n,M)\mathbb{E}\Big(\int_{0}^{T\wedge\tau_M}\|w_n(t)\|_{L^2(\mathcal{O})}^{2}dt\Big).
\end{align}
Furthermore, 
we obtain that
      \begin{align*}
\mathbb{E}\Big(\sup_{t\in[0,T\wedge\tau_M]}\|w_n({t})\|_{L^2(\mathcal{O})}^2\Big) &\leq C( \alpha,\beta,\gamma,n,M)\mathbb{E}\Big(\int_{0}^{T\wedge\tau_M}\|w_n(t)\|_{L^{2}(\mathcal{O})}^2dt\Big)\\&\leq C \mathbb{E}\Big(\int_{0}^{T}\sup_{s\in[0,t\wedge\tau_M]}\|w_n(s)\|_{L^{2}(\mathcal{O})}^2ds\Big).
\end{align*}
Then, using the Gronwall lemma, we have 
\begin{align*}
\mathbb{E}\Big(\sup_{t\in[0,T\wedge\tau_M]}\|w_n({t})\|_{L^2(\mathcal{O})}^2\Big)=0,
\end{align*}
which implies that $\|{u_n^{N_1}}-{u_n^{N_2}}\|_{L^2(\mathcal{O})}=0,$ for all $t\in[0,T\wedge\tau_M]$, $\mathbb{P}$-a.s. Letting $M\to\infty$ gives $\tau_M\uparrow\tau$ and therefore ${u_n^{N_1}}={u_n^{N_2}}$ on $[0,\tau)$, $\mathbb P$-a.s.
   \end{proof}

\subsection{Maximal solution}
   Next, we construct the maximal local solution by   using the similar idea from \cite[Theorem~3.28]{kuehn2020pathwise} and \cite[Section~4]{chen2025well}.   
   Actually, the maximal solution is defined in the following.
   \begin{defn}(Maximal solution) A maximal solution to (\ref{equ2}) is a pair $(u_n,\tau^*)$ such that for each $n\in\mathbb{N}$, $u_n(\cdot \wedge{\tau_n})$ is a local solution in the sense of Definition~\ref{def1}  and ${\tau_n}\to \tau ^*$ increasingly and $\mathbb{P}$-a.s.
   	\[
   	\sup_{t\in [0,\tau _n]}\|u_n(t)\|_{L^2(\mathcal{O})}\ge n,~~\text{on} ~\{\tau^*<\infty\}.
   	\]
   	If $\tau ^*=\infty$,     $\mathbb{P}$-a.s., the solution is called global. 
   \end{defn}

   Let $\{u_n^N(\cdot\wedge\tau_n^N)\}_{N\in\mathbb N}$ be the family of local strong solutions to the equation (\ref{equ2}), where for each $N\in\mathbb N$ the stopping time is defined by
\[
\tau_n^N:=\inf\{t\ge0:\|u_n^N(t)\|_{L^2(\mathcal O)}\ge N\}.
\] By pathwise uniqueness of local solutions, we conclude that $\{{\tau_n^N}\}_{N\in\mathbb{N}}$ is an increasing sequence of stopping times and for $N_1<N_2$,
\[
u_n^{N_1}(t)=u_n^{N_2}(t)\quad\text{for all }t\in[0,\tau_{N_1}),\ \mathbb P\text{-a.s.}
\]
   Denote $\tau^*:=\lim
   _{N\to\infty}{\tau_n^N}$. Let $(u_n,\tau^*)$ be the stochastic process defined by \[
   u_n(t):=u_n^N(t)~\text{on}~t\in[0,{\tau_n^N}),\ \mathbb P\text{-a.s.}
   \]
   Hence, $u_n$ is $H_n$-valued with $\mathbb P$-a.s. continuous paths on $[0,\tau^*)$ and $(u_n,\tau^*)$ is a local strong solution of the equation (\ref{equ2}). 
Since $\{\tau^*<\infty\}\subset\{{\tau_n^N}<\infty\}$ for every fixed $N$, $\sup_{t\in[0,{\tau_n^N}]}\|u_n^N(t)\|_{L^2(\mathcal{O})}\ge N$ on $\{\tau^*<\infty\}$.
It follows that on $\{\tau^*<\infty\}$,
\[
\sup_{0\le t<\tau^*}\|u_n(t)\|_{L^2(\mathcal O)}
\ge \sup_{0\le t\le{\tau_n^N}}\|u_n^N(t)\|_{L^2(\mathcal O)}\ge N.\]
Therefore, $\tau^*$ is the maximal existence time. Together with Theorem~\ref{thm3.2}, $(u_n,\tau^*)$
is the pathwise  unique maximal local strong solution to \eqref{equ2}. 
 
  \subsection{Uniform Estimates}
  In this subsection, we give the key uniform estimates of this paper.
\begin{lem}\label{lem3}
 Let $d\ge3,\,\alpha,\,\beta >1, \,\gamma\ge 2$ and $\alpha+1>\beta$. 	Under Assumption \ref{ass}, for $n\in\mathbb N$ and $N>0$, let  $(u_n,{\tau_n^N})$ be a pathwise unique local strong solution to (\ref{equ2}). If \begin{align}\label{condition}
     \max\{3,\gamma,\alpha+1\}<\frac{2}{d+2}(\alpha-1+\beta)+2,
 \end{align}
    then for all $T>0$, there exists a positive  constant $C(T)$ independent of $n$ and $N$  such that
    \begin{align}
\sup_{n\ge1}\mathbb{E}\Big(\sup_{t\in[0,T\wedge{\tau_n^N}]}\|u_n(t)\|_{L^2(\mathcal{O})}^2\Big)\leq C(T),
   \\
\sup_{n\ge1}\mathbb{E}\Big(\int_{0}^{T\wedge{\tau_n^N}}\|\nabla u_n(t)\|_{L^2(\mathcal{O})}^2dt\Big)\leq C(T),
  \\
\sup_{n\ge1}\mathbb{E}\Big(\int_{0}^{T\wedge{\tau_n^N}}\|u_n(t)\|_{L^{\alpha+1}(\mathcal{O})}^{\alpha+1}\|u_n(t)\|_{L^{\beta}(\mathcal{O})}^{\beta}dt\Big)\leq C(T).
    \end{align}
    Moreover,  for any $k\in[1,\frac{2}{d+2}(\alpha-1+\beta)+2)$, 
   there exists a positive  constant $C(T,k)$  independent of $n$ and $N$  such that
 \begin{align}
\sup_{n\ge1}\mathbb{E}\Big(\int_{0}^{T\wedge{\tau_n^N}}\|u_n(t)\|_{L^{k}(\mathcal{O})}^{k}dt\Big)\leq C(T,k).
 \end{align}
\end{lem}
\begin{proof}
Applying It$\hat{\text{o}}$ formula to $\|u_n\|_{L^2(\mathcal{O})}^2$, we obtain for $t\in [0,T]$ that
\begin{align}\label{1L^2}
\nonumber&\|u_n({t\wedge{\tau_n^N}})\|_{L^2(\mathcal{O})}^2-\|u_n(0)\|_{L^2(\mathcal{O})}^2\\
\nonumber=&2\int_{0}^{t\wedge{\tau_n^N}}\big(u_n(s), \Pi_n\Delta u_n(s)\big)_{L^2(\mathcal{O})}ds+2\int_{0}^{t\wedge{\tau_n^N}}\big(u_n(s), \Pi_n\sigma (u_n(s))dW(s)\big)_{L^2(\mathcal{O})}\\
\nonumber&+2\int_{0}^{t\wedge{\tau_n^N}}\big(u_n(s), \Pi_nu_n^\alpha(s)(1-\int_\mathcal{O}u_n^{\beta}(s)dx)\big)_{L^2(\mathcal{O})}ds\\\nonumber
&- 2 \int_{0}^{t\wedge{\tau_n^N}}\big(u_n(s), \Pi_n\nabla  \cdot (u_n(s) \nabla c_n(s))\big)_{L^2(\mathcal{O})}ds+ \int_{0}^{t\wedge{\tau_n^N}}\|\Pi_n\sigma (u_n(s))\|_{L_2(U_0,L^2(\mathcal{O}))}^2ds\\
\nonumber\leq&2\int_{0}^{t\wedge{\tau_n^N}}\big(u_n(s), \Delta u_n(s)\big)_{L^2(\mathcal{O})}ds+2\int_{0}^{t\wedge{\tau_n^N}}\big(u_n(s), \Pi_n\sigma (u_n(s))dW(s)\big)_{L^2(\mathcal{O})}\\
\nonumber&+2\int_{0}^{t\wedge{\tau_n^N}}\int_{\mathcal{O}}u^{\alpha+1}_n(s)dxds-2\int_{0}^{t\wedge{\tau_n^N}}\int_{\mathcal{O}}u^{\alpha+1}_n(s)dx\int_{\mathcal{O}}u^{\beta}_n(s)dxds\\
&- 2 \int_{0}^{t\wedge{\tau_n^N}}\big(u_n(s), \nabla  \cdot (u_n(s) \nabla c_n(s))\big)_{L^2(\mathcal{O})}ds+ \int_{0}^{t\wedge{\tau_n^N}}\|\sigma (u_n(s))\|_{L_2(U_0,L^2(\mathcal{O}))}^2ds,
\end{align}
where since $u_n\in H_n$, we used that $\Pi_n$ is self-adjoint and contractive on $L^2(\mathcal O)$.
 Recalling (\ref{equ2}) and using (\ref{c}), we get 
\begin{align*}
    &- 2  \int_{\mathcal{O}}u_n(s)\nabla  \cdot (u_n(s) \nabla c_n(s))dx= \int_{\mathcal{O}}u^2_n(s)(u_n(s)-c_n(s))dx\\\leq&  \|u_n\|_{L^3(\mathcal{O})}^3+ \|u_n^2\|_{L^\frac{3}{2}(\mathcal{O})}\|c_n\|_{L^3(\mathcal{O})}\leq C\|u_n\|_{L^3(\mathcal{O})}^3,
\end{align*}
substituting the above formula into (\ref{1L^2}) and using (\ref{ass3}) in  Assumption \ref{ass}, we have 
\begin{align}\label{L^2}
    \nonumber&\|u_n({t\wedge{\tau_n^N}})\|_{L^2(\mathcal{O})}^2 +2\int_{0}^{t\wedge{\tau_n^N}}\|\nabla u_n(s)\|_{L^2(\mathcal{O})}^2ds+2\int_{0}^{t\wedge{\tau_n^N}}\int_{\mathcal{O}} u^{\alpha+1}_n(s)dx\int_{\mathcal{O}} u_n^{\beta}(s)dxds\\ \nonumber
   \leq&\|u_n(0)\|_{L^2(\mathcal{O})}^2+2\Big|\int_{0}^{t\wedge{\tau_n^N}}(u_n(s), \Pi_n\sigma (u_n(s))dW(s))_{L^2(\mathcal{O})}\Big|+2\int_{0}^{t\wedge{\tau_n^N}}\int_{\mathcal{O}}u^{\alpha+1}_n(s)dxds\\ 
&+C\int_{0}^{t\wedge{\tau_n^N}}\int_{\mathcal{O}}u^3_n(s)dxds+C\int_{0}^{t\wedge{\tau_n^N}}\|u_n(s)\|_{L^\gamma(\mathcal{O})}^\gamma ds.
\end{align}
Taking the supremum over $[0,T]$ and the expectation from both sides of (\ref{L^2}), we obtain the following
\begin{align}\label{25L^2}
\nonumber&\mathbb{E}\Big(\sup_{t\in[0,T]}\|u_n({t\wedge{\tau_n^N}})\|_{L^2(\mathcal{O})}^2\Big) +2\mathbb{E}\Big(\int_{0}^{T\wedge{\tau_n^N}}\|\nabla u_n(t)\|_{L^2(\mathcal{O})}^2dt\Big)\\\nonumber&+2\mathbb{E}\Big(\int_{0}^{T\wedge{\tau_n^N}}\int_{\mathcal{O}} u^{\alpha+1}_n(t)dx\int_{\mathcal{O}} u_n^{\beta}(t)dxdt\Big)\\\nonumber \leq&\mathbb{E}\big(\|u_n(0)\|_{L^2(\mathcal{O})}^2\big)+2\mathbb{E}\Big(\int_{0}^{T\wedge{\tau_n^N}}\|u_n(t)\|_{L^{\alpha+1}(\mathcal{O})}^{\alpha+1}dt\Big)+C\mathbb{E}\Big(\int_{0}^{T\wedge{\tau_n^N}}\|u_n(t)\|_{L^{3}(\mathcal{O})}^{3}dt\Big)\\&+C\mathbb{E}\Big(\int_{0}^{T\wedge{\tau_n^N}}\|u_n(t)\|_{L^\gamma(\mathcal{O})}^\gamma dt\Big)+2\mathbb{E}\Big(\sup_{t\in[0,T]}\Big|\int_{0}^{t\wedge{\tau_n^N}}(u_n(s), \Pi_n\sigma (u_n(s))dW(s))_{L^2(\mathcal{O})}\Big|\Big).
\end{align}
For the last term on the right of (\ref{25L^2}), by the Burkholder-Davis-Gundy inequality and Young's inequality, using Assumption (\ref{ass3}), we get
\begin{align}\label{6L^2}
    \nonumber &2\mathbb{E}\Big(\sup_{t\in[0,T]}\Big|\int_{0}^{t\wedge{\tau_n^N}}(u_n(s), \Pi_n\sigma (u_n(s))dW(s))_{L^2(\mathcal{O})}\Big|\Big)\\\nonumber
\leq&C\mathbb{E}\Big(\int_{0}^{T\wedge{\tau_n^N}}\|u_n(t)\|_{L^2(\mathcal{O})}^2\|\Pi_n\sigma (u_n(t))\|_{L_2(U_0,L^2(\mathcal{O}))}^2dt\Big)^\frac{1}{2}\\\nonumber
\leq&C\mathbb{E}\Big(\Big(\sup_{t\in[0,T]}\|u_n(t\wedge{\tau_n^N})\|_{L^2(\mathcal{O})}^2\Big)^\frac{1}{2}\Big(\int_{0}^{T\wedge{\tau_n^N}}\|\sigma (u_n(t))\|_{L_2(U_0,L^2(\mathcal{O}))}^2dt\Big)^\frac{1}{2}\Big)\\
    \leq&\frac{1}{2}\mathbb{E}\Big(\sup_{t\in[0,T]}\|u_n(t\wedge{\tau_n^N})\|_{L^2(\mathcal{O})}^2\Big)+C\mathbb{E}\Big(\int_{0}^{T\wedge{\tau_n^N}}\|u_n(t)\|_{L^\gamma(\mathcal{O})}^\gamma dt\Big).
\end{align}
Substituting (\ref{6L^2}) into (\ref{25L^2}), we see that 
\begin{align}\label{35L^2}
\nonumber&\frac{1}{2}\mathbb{E}\Big(\sup_{t\in[0,T]}\|u_n({t\wedge{\tau_n^N}})\|_{L^2(\mathcal{O})}^2\Big) +2\mathbb{E}\Big(\int_{0}^{T\wedge{\tau_n^N}}\|\nabla u_n(t)\|_{L^2(\mathcal{O})}^2dt\Big)\\\nonumber&+2\mathbb{E}\Big(\int_{0}^{T\wedge{\tau_n^N}}\int_{\mathcal{O}} u^{\alpha+1}_n(t)dx\int_{\mathcal{O}} u_n^{\beta}(t)dxdt\Big)\\\nonumber \leq&\mathbb{E}\big(\|u_n(0)\|_{L^2(\mathcal{O})}^2\big)+2\mathbb{E}\Big(\int_{0}^{T\wedge{\tau_n^N}}\|u_n(t)\|_{L^{\alpha+1}(\mathcal{O})}^{\alpha+1}dt\Big)+C\mathbb{E}\Big(\int_{0}^{T\wedge{\tau_n^N}}\|u_n(t)\|_{L^{3}(\mathcal{O})}^{3}dt\Big)\\&+C\mathbb{E}\Big(\int_{0}^{T\wedge{\tau_n^N}}\|u_n(t)\|_{L^\gamma(\mathcal{O})}^\gamma dt\Big).
\end{align}
Next, we use Lemma~\ref{lem1} together with interpolation and Young's inequality to derive an absorbable
$L^k(\mathcal{O})$-estimate. 
\begin{Claim}
Let $\alpha>1,\,\beta>1$ and $\alpha+1>\beta$. We   define
 \[K=\{k:\alpha+1\leq k<\frac{2}{d+2}(\alpha-1+\beta)+2\},\]
 for any $k\in K$, $\varepsilon_1,\varepsilon_2>0$, there exists a constant
$C>0$ independent of $n,\,N$  such that for a.e.\ $t\in[0,T\wedge{\tau_n^N}]$
\begin{align}\label{f10}
\|u_n(t)\|_{L^{k}(\mathcal{O})}^{k}
\leq\;& \varepsilon_1\|\nabla u_n(t)\|_{L^{2}(\mathcal{O})}^{2}
+\varepsilon_2\|u_n(t)\|_{L^{\alpha+1}(\mathcal{O})}^{\alpha+1}\|u_n(t)\|_{L^{\beta}(\mathcal{O})}^{\beta}
+C\|u_n(t)\|_{L^{2}(\mathcal{O})}^{2}+C.
\end{align} 
\end{Claim}
\begin{proof}[Proof of Claim]
By Lemma \ref{lem1}, we have that   for any $k\in K$, if $r\ge1$ satisfying 
\begin{align}\label{d4}
    \frac{d(k-2)}{2}<r<k<\frac{2d}{d-2},
\end{align}
for any $\varepsilon_1>0$, there exists the constant $C({\varepsilon_1,d})>0$ such that  
\begin{align}\label{dd7}
    \|u_n\|_{L^k(\mathcal{O})}^k\leq \varepsilon_1\|\nabla u_n\|_{L^2(\mathcal{O})}^2+C(\varepsilon_1,d)\| u_n\|_{L^r(\mathcal{O})}^{\zeta}+C\| u_n\|_{L^2(\mathcal{O})}^2,
\end{align}
where 
\[\lambda = \frac{\frac{1}{r} - \frac{1}{k}}{\frac{1}{r} - \frac{d-2}{2d}}\in(0,1),\,\zeta =\frac{2k(1-\lambda)}{2-\lambda k}.\]
Then, using the interpolation inequality, we get
\begin{align}\label{interpolation}
    \| u_n\|_{L^r(\mathcal{O})}^{\zeta}\leq (\| u_n\|_{L^{\alpha+1}(\mathcal{O})}^{\alpha+1}\| u_n\|_{L^\beta(\mathcal{O})}^\beta)^\frac{\zeta\theta}{\alpha+1}\| u_n\|_{L^\beta(\mathcal{O})}^{\zeta(1-\theta-\frac{\theta\beta}{\alpha+1})}
\end{align}
with $\theta=\frac{\frac{1}{\beta}-\frac{1}{r}}{\frac{1}{\beta}-\frac{1}{\alpha+1}}\in (0,1)$.
Furthermore, by the definition of $K$ and (\ref{d4}), we can take $r=\frac{\alpha+1+\beta}{2}$ such that $1-\theta-\frac{\theta\beta}{\alpha+1}=0$.
If for every $k\in K$, it holds $\frac{\zeta\theta}{\alpha+1}<1$, then by Young's inequality, for any $\varepsilon_2>0$, there exists a positive constant $C(\varepsilon_2)$ such that 
\begin{align}\label{dd6}
     C(\varepsilon_1,d)\| u_n\|_{L^r(\mathcal{O})}^{\zeta}\leq \varepsilon_2\| u_n\|_{L^{\alpha+1}(\mathcal{O})}^{\alpha+1}\| u_n\|_{L^\beta(\mathcal{O})}^\beta+C(\varepsilon_1,\varepsilon_2,\alpha,\beta,d).
\end{align}
Therefore, substituting (\ref{dd6}) into (\ref{dd7}), we obtain that the estimate (\ref{f10}) holds.
It is necessary to verify that the conditions $k\in K,\,r=\frac{\alpha+1+\beta}{2}$ under which $\frac{\zeta\theta}{\alpha+1}<1$ holds.
 Indeed, by the definition of $\zeta,\,\theta,\,\lambda$, we know that $\frac{\zeta\theta}{\alpha+1}<1$ is equivalent  to 
\begin{align}\label{3.25}
    (1-\frac{r(d-2)}{2d})\frac{k-\alpha-1}{\beta}<(\frac{1}{2}-\frac{d-2}{2d}-\frac{\alpha-1}{2\beta})(k-r).
\end{align}
Define
\begin{align}\label{A}
    A_0=\frac{1}{2}-\frac{d-2}{2d}-\frac{\alpha-1}{2\beta},\,A_1=\frac{k-\alpha-1}{\beta}.
\end{align}

The condition (\ref{d4}) implies    $1-\frac{r(d-2)}{2d}>0$ and $k-r>0$. By the definition of $K$, we know that  $A_0>0$ and  $A_1\ge 0$.  Then, by some simple computations, 
we obtain (\ref{3.25}) is equivalent to
\begin{align}\label{d2}
    2B>D,
\end{align}
where
\[
B:=\frac{A_0}{2}-A_1(\frac{1}{2}-\frac{d-2}{4d}),~D:=A_0\Big(\frac{\alpha+\beta-1}{2}-(k-2)\Big)-\frac{A_1(d-2)}{4d}\Big(\alpha+\beta-1\Big).
\]
Since 
\begin{align}\label{d1}
  k-2<\frac{2}{d+2}(\alpha-1+\beta),
\end{align} which is equivalent to
\begin{align*}
    \frac{k-\alpha-1}{\beta}\Big(1-\frac{d-2}{2d}\Big)<\frac{1}{2}-\frac{d-2}{2d}-\frac{\alpha-1}{2\beta},
\end{align*}
 it shows that $B>0$. 
In the following, we use (\ref{d1}) to prove  $D>0$. 
Indeed, 
\begin{align*}
   \nonumber D=&A_0\Big(\frac{\alpha+\beta-1}{2}-(k-2)\Big)-\frac{A_1(d-2)}{4d}\Big(\alpha+\beta-1\Big)\\
   \nonumber >&A_0\Big(\frac{\alpha+\beta-1}{2}-\frac{2}{d+2}(\alpha-1+\beta)\Big)-\frac{A_1(d-2)}{4d}\Big(\alpha+\beta-1\Big)\\\nonumber
    =&\frac{A_0(d-2)}{d+2}\frac{\alpha+\beta-1}{2}-\frac{A_1(d-2)}{4d}\Big(\alpha+\beta-1\Big)>0.
\end{align*}
Then (\ref{d2})  is equivalent to 
\begin{align}\label{A0}
    \nonumber 2>&\frac{D}{B}=\frac{\Big(-A_0-\frac{(\alpha+\beta-1)(d-2)}{4d\beta}\Big)\Big(k-1\Big)+\frac{\alpha(d-2)}{2d\beta}\frac{\alpha+\beta-1}{2}+A_0\frac{\alpha+\beta-1}{2}}{-\frac{1}{\beta}\Big(\frac{1}{2}-\frac{d-2}{4d}\Big)\Big(k-1\Big)+\frac{A_0}{2}+\frac{\alpha}{\beta}\Big(\frac{1}{2}-\frac{d-2}{4d}\Big)}\\
    \nonumber=&\frac{-\frac{1}{\beta}\Big(\beta-(\alpha-1)\Big)\Big(\frac{1}{2}-\frac{d-2}{4d}\Big)\Big(k-1\Big)+\Big(\beta-(\alpha-1)\Big)\Big(\frac{A_0}{2}+\frac{\alpha}{\beta}\Big(\frac{1}{2}-\frac{d-2}{4d}\Big)\Big)}{-\frac{1}{\beta}\Big(\frac{1}{2}-\frac{d-2}{4d}\Big)\Big(k-1\Big)+\frac{A_0}{2}+\frac{\alpha}{\beta}\Big(\frac{1}{2}-\frac{d-2}{4d}\Big)}\\
    =&\beta-(\alpha-1).
\end{align}
Combining (\ref{d1}), (\ref{A0}), and $A_0>0,\,A_1\ge0$, we infer that for any $k\in K$,
it holds $\frac{\zeta\theta}{\alpha+1}<1$. Therefore, for any $k\in K$, the estimate (\ref{f10}) holds. This finishes the proof of \textbf{Claim}.
\end{proof}
Let $m:=\max\{3,\gamma,\alpha+1\}$. Since $\mathcal O$ is a bound domain, for any $p\le m$,
H\"older's inequality and Young's inequality imply\begin{align}\label{dd8}
    \|u_n(t)\|_{L^p(\mathcal {O})}^{p}\le C\big(1+\|u_n(t)\|_{L^m(\mathcal {O})}^{m}\big).
\end{align}
Hence, in order to control the three terms $\|u_n\|_{L^3(\mathcal {O})}^3$, $\|u_n\|_{L^\gamma(\mathcal {O})}^\gamma$ and $\|u_n\|_{L^{\alpha+1}(\mathcal {O})}^{\alpha+1}$,
it is sufficient to control $\|u_n\|_{L^m(\mathcal {O})}^{m}$.
By the condition (\ref{condition}) on Theorem \ref{thm3.2}, we know $m\in K$. Applying \eqref{f10} with $k=m$, integrating over $[0,T\wedge{\tau_n^N}]$ and substituting into
\eqref{35L^2}, we choose $\varepsilon_1,\varepsilon_2>0$ small so that the
$\varepsilon_1$- and $\varepsilon_2$-terms are absorbed by the left-hand side of \eqref{35L^2}.
Therefore, we infer  that 
\begin{align}\label{7L^2}
     \nonumber&\frac{1}{2}\mathbb{E}\Big(\sup_{t\in[0,T]}\|u_n(t\wedge{\tau_n^N})\|_{L^2(\mathcal{O})}^2\Big) +\mathbb{E}\Big(\int_{0}^{T\wedge{\tau_n^N}}\|\nabla u_n(t)\|_{L^2(\mathcal{O})}^2dt\Big)\\
     \nonumber&+\mathbb{E}\Big(\int_{0}^{T\wedge{\tau_n^N}}\int_{\mathcal{O}} u^{\alpha+1}_n(t)dx\int_{\mathcal{O}} u_n^{\beta}(t)dxdt\Big)\\  \nonumber\leq&\mathbb{E}\big(\|u_n(0)\|_{L^2(\mathcal{O})}^2\big)+C\mathbb{E}\int_{0}^{T\wedge{\tau_n^N}}(1+\|u_n(t)\|_{L^2(\mathcal{O})}^2)dt\\\leq &\sup_{n\ge1}\mathbb{E}\big(\|u_n(0)\|_{L^2(\mathcal{O})}^2\big)+C\mathbb{E}\int_{0}^{T}(1+\sup _{s\in[0,t]}\|u_n(s\wedge{\tau_n^N})\|_{L^2(\mathcal{O})}^2)ds.
\end{align}
Using the Gronwall lemma yields that 
\begin{align}\label{dd9}
    \nonumber&\frac{1}{2}\mathbb{E}\Big(\sup_{t\in[0,T]}\|u_n({t\wedge{\tau_n^N}})\|_{L^2(\mathcal{O})}^2\Big) +\mathbb{E}\Big(\int_{0}^{T\wedge{\tau_n^N}}\|\nabla u_n(t)\|_{L^2(\mathcal{O})}^2dt\Big)\\
&+\mathbb{E}\Big(\int_{0}^{T\wedge{\tau_n^N}}\int_{\mathcal{O}} u^{\alpha+1}_n(t)dx\int_{\mathcal{O}} u_n^{\beta}(t)dxdt\Big)\leq C\Big(T,\mathbb{E}(\|u_0\|_{L^2(\mathcal{O})}^2,\alpha,\gamma,d\Big).
\end{align}

Consequently, by (\ref{f10}), (\ref{dd8}) and (\ref{dd9}), we have that for any $k\in[1,\frac{2}{d+2}(\alpha-1+\beta)+2)$ there exists $C>0$ independent of $n,N$ such that
\begin{align}
    \mathbb{E}\Big(\int_{0}^{T\wedge{\tau_n^N}}\| u_n(t)\|_{L^k(\mathcal{O})}^kdt\Big)\leq C\Big(T,\mathbb{E}(\|u_n(0)\|_{L^2(\mathcal{O})}^2,\alpha,\gamma,d,k\Big).
\end{align}
Thus, this completes the proof of Lemma \ref{lem3}.
\end{proof}

\subsection{Global Strong Solution}
   \begin{thm}
   Let $d\ge 3$,  the constant $\alpha>1,\beta>1,\gamma\ge2$ and $ \alpha+1>\beta$. Under Assumption \ref{ass}, if \[\max\{3,\gamma,\alpha+1\}<\frac{2}{d+2}(\alpha-1+\beta),\]  then  
    for each $n\in \mathbb{N}$ and for every $T>0$, the equation  (\ref{equ2}) admits the pathwise  unique global nonnegative solution $u_n$ on $[0,T]$. There exists positive  constant $C(T)$, $C(T,k)$ independent of $n$   such that
    \begin{align}\label{a1}
\sup_{n\ge1}\mathbb{E}\Big(\sup_{t\in[0,T]}\|u_n(t)\|_{L^2(\mathcal{O})}^2\Big)\leq C(T),
   \\\label{a2}
\sup_{n\ge1}\mathbb{E}\Big(\int_{0}^{T}\|\nabla u_n(t)\|_{L^2(\mathcal{O})}^2dt\Big)\leq C(T),
  \\\label{a3}
\sup_{n\ge1}\mathbb{E}\Big(\int_{0}^{T}\|u_n(t)\|_{L^{\alpha+1}(\mathcal{O})}^{\alpha+1}\|u_n(t)\|_{L^{\beta}(\mathcal{O})}^{\beta}dt\Big)\leq C(T),\\\label{kk}
\sup_{n\ge1}\mathbb{E}\Big(\int_{0}^{T}\|u_n(t)\|_{L^{k}(\mathcal{O})}^{k}dt\Big)\leq C(T,k),
 \end{align}
  where  $k\in[1,\frac{2}{d+2}(\alpha-1+\beta)+2)$.
   \end{thm}
   \begin{proof} 
       By Chebyshev's inequality and  Lemma  \ref{lem3} and the definition of stopping time ${\tau_n^N}$, we know that 
       \begin{align*}
       \nonumber\mathbb{P}\Big\{{\tau_n^N}\leq T\Big\}&=\mathbb{P}\Big\{\sup_{t\in[0,T]}\|u_n({t\wedge{\tau_n^N}})\|_{L^2(\mathcal{O})}\ge N\Big\}\\&\leq \frac{1}{N^2}\mathbb{E}\Big(\sup_{t\in[0,T]}\|u_n({t\wedge{\tau_n^N}})\|_{L^2(\mathcal{O})}^2\Big)
\leq \frac{C(T)}{N^2}\to 0,~(N\to \infty).
       \end{align*}
       Since $\tau_n^N\uparrow\tau^*$, we have
	\[
	\{\tau^*\le T\}=\bigcap_{N\in\mathbb N}\{\tau_n^N\le T\}.
	\]
	Therefore, by continuity of probability,
	\[
	\mathbb P(\tau^*\le T)=\lim_{N\to\infty}\mathbb P(\tau_n^N\le T)=0.
	\]
	As $T>0$ is arbitrary, it follows that $\mathbb P(\tau^*<\infty)=0$, i.e. $\tau^*=\infty$ $\mathbb P$-a.s.
	Hence $u_n(t)$ is defined globally on $t\in [0,\infty)$.
	The nonnegativity and uniqueness follow from  Theorem~\ref{thm3.1} and Theorem~\ref{thm3.2}.
    
    Finally, taking the limit $N\to\infty$ and observing that $T\wedge{\tau_n^N}\to T$, by the monotone convergence theorem, we infer that there exists a positive  constant $C(T)$ independent of $n$ and $N$  such that the uniform estimates \eqref{a1}-\eqref{kk} hold true.
   \end{proof}

   \section{Global Martingale solution}
  In this section, we rely on uniform estimates for the Galerkin approximations to establish the tightness of the corresponding family of laws on a suitable non-metric topological space. We then apply Jakubowski's version of the Skorokhod theorem and use Vitali's convergence theorem to pass to the limit in the nonlinear terms, which allows us to construct a global martingale solution to the stochastic Keller-Segel system \eqref{equ1}.

   \subsection{Tightness of the sequence of the Galerkin solutions}
    In Appendix, we introduce a general  topological $(\mathcal{Z},\mathcal{T})$ and  establish the corresponding compactness criterion and tightness criterion. We now specify these abstract spaces to our setting. Let $s>\frac{d}{2}+1$, $2<k<\frac{2}{d+2}(\alpha-1+\beta)+2$, \[\mathbb{U}:=H^s(\mathcal{O}),\quad\mathbb H:=L^2(\mathcal{O}),\quad\mathbb V:=H^1(\mathcal{O}),\quad E:=L^k(\mathcal{O}).\]  We define   the space 
        \[
        Z_T:=C([0,T]; H^s(\mathcal{O})')\cap L^2_w(0,T; H^1(\mathcal{O}))\cap L^k(0,T; L^k(\mathcal{O}))\cap C([0,T]; L^2_w(\mathcal{O})),
        \]
        with the topology $\mathcal{T}$ which is the supremum of the corresponding topologies.

        \begin{lem}\label{lem4.1}
            The sequence of probability measures $\{\mathcal{L}(u_n):n\in\mathbb{N}\}$ is tight on $(Z_T,\mathcal{T})$, where $\mathcal{L}(u_n)$ denotes the law of  $u_n$ as a $Z_T$-valued random variable.
        \end{lem}
        \begin{proof}
        We  apply Lemma \ref{tight} to prove $\{\mathcal{L}(u_n):n\in\mathbb{N}\}$ is tight on $(Z_T,\mathcal{T})$. 
              Combining the estimates (\ref{a1}), (\ref{a2}) and (\ref{kk}), which are corresponding conditions (a),(b),(c), it is sufficient to prove that the sequence $\{u_n:n\in\mathbb{N}\}$ satisfies the Aldous condition in $H^s(\mathcal{O})'$.  Let $t\in [0,T]$ and $\phi\in H^s(\mathcal{O})$, by (\ref{equ2}), we have 
            \begin{align}\label{b1}
    \nonumber\langle{u}_n(t),\phi\rangle=&\langle{u_n(0)},\phi\rangle+\int _0^t\Big\langle\Pi_n\Delta{u}_n(s),\phi\Big\rangle ds -  \int _0^t\Big\langle\Pi_n\nabla\cdot\big({u}_n(s)\nabla {c}_n(s)\big),\phi \Big\rangle ds \\\nonumber
    &+\int _0^t\Big\langle\Pi_n{u}_n^\alpha(s)(1-\int_\mathcal{O}{u}^{\beta}_n(s)dx),\phi \Big\rangle ds+\Big\langle\int _0^t\Pi_n\sigma({u}_n(s))d{W}(s),\phi \Big\rangle\\
:=&\langle u_n(0),\phi\rangle+\sum_{i=1}^{4}\big\langle J_i^n(t),\phi\big\rangle,
\end{align}
where $\langle \cdot,\cdot\rangle$ is the dual pairing between $ H^s(\mathcal{O})'$ and $ H^s(\mathcal{O})$.

 Next, we estimate $J_i,\,1\leq i\leq 4$ of (\ref{b1}) individually. 	Throughout the proof we repeatedly use the following elementary inequality. For any $1\le m<k<\frac{2}{d+2}(\alpha-1+\beta)+2$, by $L^k(\mathcal{O})\hookrightarrow L^m(\mathcal{O})$ and the  uniform bound (\ref{kk}),  we use the H\"older inequality in  time and the Jensen inequality in expectation to  infer that for any $t\in[0,T]$, 
 \begin{align}\label{rk}
 \nonumber&\sup_{n\ge 1}\mathbb E\Big(\int_0^t \|u_n(s)\|_{L^m(\mathcal O)}^mds\Big)
\leq C(m,k)\sup_{n\ge 1}\mathbb{E}\Big(\int_0^t\|u_n(s)\|_{L^k(\mathcal{O})}^mds\Big)\\\nonumber\leq &C(m,k)\sup_{n\ge 1}\mathbb{E}\Big[t^\frac{k-m}{k}\Big(\int_0^t\|u_n(s)\|_{L^k(\mathcal{O})}^kds\Big)^\frac{m}{k}\Big]\\\leq& C(m,k)t^\frac{k-m}{k}\sup_{n\ge 1}\Big[\mathbb{E}\Big(\int_0^t\|u_n(s)\|_{L^k(\mathcal{O})}^kds\Big)\Big]^\frac{m}{k}\le C(m,k)\, t^{\frac{k-m}{k}}.
 \end{align}

Let $({\tau_n})_{n\in\mathbb{N}}$ be a sequence of stopping times such that $0\leq{\tau_n}\leq T.$ Let $\theta>0$, using  the H$\ddot{\text{o}}$lder inequality, we have
 \begin{align*}
    &\Big|\Big\langle J_1^n({\tau_n}+\theta)-J_1^n({\tau_n}),\phi\Big\rangle \Big|= \Big|\int _{{\tau_n}}^{{\tau_n}+\theta}\Big\langle\nabla{u}_n(s),\nabla \Pi_n\phi\Big\rangle ds\Big|
     \\\leq&\int _{{\tau_n}}^{{\tau_n}+\theta}\|\nabla{u}_n(s)\|_{L^2(\mathcal{O})}\|\nabla\Pi_n \phi\|_{L^2(\mathcal{O})} ds
     \leq C\theta^\frac{1}{2}\|\nabla \phi\|_{L^2(\mathcal{O})} \Big(\int _{0}^{T}\|\nabla{u}_n(t)\|_{L^2(\mathcal{O})}^2dt\Big)^\frac{1}{2}.
 \end{align*}
By the definition of the $H^s(\mathcal{O})'$ norm and the Jensen inequality, we use the estimate (\ref{a2}) to  obtain 
 \begin{align}\label{b2}
     \nonumber& \mathbb{E}\Big(\|J_1^n({\tau_n}+\theta)-J_1^n({\tau_n})\|_{H^s(\mathcal{O})'}\Big)=\mathbb{E}\Big(\sup_{\|\phi\|_{H^s(\mathcal{O})}=1}\Big|\Big\langle J_1^n({\tau_n}+\theta)-J_1^n({\tau_n}),\phi\Big\rangle \Big|\Big)\\\leq & C\theta^\frac{1}{2}\mathbb{E}\Big[\Big(\int _{0}^{T}\|\nabla{u}_n(t)\|_{L^2(\mathcal{O})}^2dt\Big)^\frac{1}{2}\Big]\leq C_1\theta^\frac{1}{2}.
 \end{align}
 For $J_2^n(t)$. Using the embedding $ H^s(\mathcal{O})\hookrightarrow W^{1,\infty}(\mathcal{O})$, the H\"older inequality and (\ref{c}), we get
 \begin{align*}
    &\Big|\Big\langle J_2^n({\tau_n}+\theta)-J_2^n({\tau_n}),\phi\Big\rangle \Big|= \Big|\int _{{\tau_n}}^{{\tau_n}+\theta}\Big\langle{u}_n(s)\nabla {c}_n(s),\nabla  \Pi_n\phi \Big\rangle ds\Big|
     \\&\leq C\int _{{\tau_n}}^{{\tau_n}+\theta}\|u_n(s)\|_{L^2(\mathcal{O})}\|\nabla{c}_n(s)\|_{L^2(\mathcal{O})}\|\nabla \phi\|_{L^\infty(\mathcal{O})} ds
     \leq C\| \phi\|_{H^s(\mathcal{O})} \int _{{\tau_n}}^{{\tau_n}+\theta}\|{u}_n(t)\|_{L^2(\mathcal{O})}^2dt.
 \end{align*}
 Hence, by the estimate (\ref{rk}), taking $2<k_1<\frac{2}{d+2}(\alpha-1+\beta)+2$, we know 
 \begin{align}\label{b3}
      \mathbb{E}\Big(\|J_2^n({\tau_n}+\theta)-J_2^n({\tau_n})\|_{H^s(\mathcal{O})'}\Big)\leq  C\mathbb{E}\Big(\int _{{\tau_n}}^{{\tau_n}+\theta}\|{u}_n(t)\|_{L^2(\mathcal{O})}^2dt\Big)\leq C_2\theta^\frac{k_1-2}{k_1}.
 \end{align}
 In a similar way, we estimate $J_3^n$. Using  inequality (\ref{rk}) and taking  $\alpha+\beta<k_2<\frac{2}{d+2}(\alpha-1+\beta)+2$ such that
 \begin{align}\label{b4}
    \nonumber& \mathbb{E}\Big(\|J_3^n({\tau_n}+\theta)-J_3^n({\tau_n})\|_{H^s(\mathcal{O})'}\Big)\\\leq&\mathbb{E}\Big(\sup_{\|\phi\|_{H^s(\mathcal{O})}=1}\Big|\int _{{\tau_n}}^{{\tau_n}+\theta}(1-\int_\mathcal{O}{u}^{\beta}_n(s)dx)\Big\langle{u}_n^\alpha(s),\Pi_n\phi \Big\rangle  ds\Big|\Big)\\\nonumber
     \leq&C \mathbb{E}\Big(\sup_{\|\phi\|_{H^s(\mathcal{O})}=1}\int _{{\tau_n}}^{{\tau_n}+\theta}(1+\|u_n(s)\|_{L^\beta(\mathcal{O})}^\beta)\|u_n(s)\|_{L^\alpha(\mathcal{O})}^\alpha\|\phi\|_{L^\infty(\mathcal{O})} ds\Big)\\\nonumber
     \leq &C\mathbb{E}\Big(\int _{{\tau_n}}^{{\tau_n}+\theta}\|u_n(s)\|_{L^\alpha(\mathcal{O})}^\alpha+\|u_n(s)\|_{L^{k_2}(\mathcal{O})}^{\alpha+\beta} ds\Big)\\
     \leq &C \theta^\frac{k_2-\alpha}{k_2}+C \theta^\frac{k_2-\alpha-\beta}{k_2}\leq C_3\theta^\frac{k_2-\alpha-\beta}{k_2}.
 \end{align}
 For the last term $J_4^n$,  using the Burkholder-Davis-Gundy inequality and  taking $\gamma<k_3<\frac{2}{d+2}(\alpha-1+\beta)+2$, by \eqref{ass3} in Assumption \ref{ass}, we obtain
 \begin{align}\label{b5}
    \nonumber &\mathbb{E}\Big(\|J_4^n({\tau_n}+\theta)-J_4^n({\tau_n})\|_{H^s(\mathcal{O})'}\Big)\leq\mathbb{E}\Big(\sup_{\|\phi\|_{H^s(\mathcal{O})}=1}\Big|\Big\langle\int _{{\tau_n}}^{{\tau_n}+\theta}\Pi_n\sigma({u}_n(s))d{W}(s),\phi \Big\rangle\Big|^2\Big)\\\nonumber\leq &C\mathbb{E}\Big(\big\|\int _{{\tau_n}}^{{\tau_n}+\theta}\Pi_n\sigma({u}_n(s))d{W}(s)\big\|_{L^2(\mathcal{O})}^2\Big)
     \leq C \mathbb{E}\Big(\int _{{\tau_n}}^{{\tau_n}+\theta}\|\sigma(u_n(s))\|_{L_2(U_0,L^2(\mathcal{O}))}^2 ds\Big)
     \\\leq& C \mathbb{E}\Big(\int _{{\tau_n}}^{{\tau_n}+\theta}\|u_n(s)\|_{L^\gamma(\mathcal{O})}^\gamma ds\Big)\leq C_4\theta^\frac{k_3-\gamma}{2k_3}.
 \end{align}
  Let $\eta>0$ and $\varepsilon>0$. By the Chebyshev inequality and  estimates (\ref{b2}), (\ref{b3}), (\ref{b4}) and (\ref{b5}) we obtain 
 \begin{align}
     \nonumber&\mathbb{P}\Big(\{\|J_i^n({\tau_n}+\theta)-J_i^n({\tau_n})\|_{H^s(\mathcal{O})'}\ge\eta\}\Big)\leq \frac{1}{\eta}  \mathbb{E}\Big(\|J_i^n({\tau_n}+\theta)-J_i^n({\tau_n})\|_{H^s(\mathcal{O})'}\Big)\leq C_i\frac{\theta^{\rho_i}}{\eta},
 \end{align}
 where $\rho_1=\frac12$, $\rho_2=\frac{k_1-2}{k_1}$, $\rho_3=\frac{k_2-\alpha-\beta}{k_2}$, $\rho_4=\frac{k_3-\gamma}{2k_3}$.
	Choosing $
	\delta:=\min_{1\le i\le 4}\Big(\frac{\eta\varepsilon}{4C_i}\Big)^{1/\rho_i},
	$
 we infer 
 \begin{align}
     \sup_{n\in\mathbb{N}} \sup_{0<\theta<\delta} \mathbb{P}\Big(\{\|J_i^n({\tau_n}+\theta)-J_i^n({\tau_n})\|_{H^s(\mathcal{O})'}\ge\eta\}\Big)\leq\frac{\varepsilon}{4}.
 \end{align}
 This shows that the Aldous condition holds for each term $J_i^n,~1\leq i\leq 4.$ Therefore, combining estimates (\ref{a1}) and (\ref{a2}), we infer that the sequence $\{u_n:n\in\mathbb{N}\}$ satisfies the tightness criterion. This completes the proof.
        \end{proof}
        
\subsection{Application of the Jakubowski-Skorokhod Theorem}

      Let $W_n:=W$ for all $n\in\mathbb{N},$ be a $U$-valued $Q$-Wiener process, the set $\{\mathcal{L}(W_n):n\in\mathbb{N}\}$ is tight on $C([0,T];U)$.  Thus the set  $\{\mathcal{L}\big((u_n,W_n)\big):n\in\mathbb{N}\}$ is tight on $Z_T\times C([0,T];U).$  By the Jakubowski's version of  the Skorokhod Theorem for non-metric space in \cite[ Corollary 3.12]{brzezniak2013existence} and \cite[Corollary 13]{dhariwal2019global}, there exists a subsequence $\{(u_n,W_n):n\in\mathbb{N}\}$, which is not relabeled, a probability space $(\tilde{\Omega} ,\tilde{\mathcal{F}},\tilde{\mathbb{P}}  )$, and  $Z_T\times C([0,T];U)$-valued random variables $(\tilde{u},\tilde{W})$ and $(\tilde{u}_n,\tilde{W}_n)$ with $n\in\mathbb{N}$ such that 
        	\[
	\mathcal L\big((\tilde u_n,\tilde W_n)\big)=\mathcal L\big((u_n,W_n)\big)
	\quad\text{on }\mathcal B\big(Z_T\times C([0,T];U)\big),
	\]and
        \[
        (\tilde{u}_n,\tilde{W}_n)\to (\tilde{u},\tilde{W})~\text{in} ~Z_T\times C([0,T];U),~\tilde{\mathbb{P}}\text{-a.s.}\,~\text{as}~n\to\infty.
        \]
        Note that since $\mathcal{B}(Z_T)\otimes\mathcal{B}(C([0,T];U))\subset\mathcal{B}(Z_T\times C([0,T];U))$, $\tilde u_n$ and $\tilde u$ are $Z_T$-valued Borel random variables.
        
        	Let $\{\tilde{\mathcal F}_t\}_{t\in[0,T]}$ be a complete and right-continuous   filtration
	generated with $(\tilde u,\tilde W)$.
Denote $\tilde{\mathfrak{A}}:= (\tilde{\Omega} ,\tilde{\mathcal{F}}, \{\tilde{\mathcal{F}}_{t}\} _{t\ge 0}  ,\tilde{\mathbb{P}}  )$ as a stochastic basis. Similarly, based on $(\tilde{u}_n,\tilde{W}_n)$, we denote $\tilde{\mathfrak{A}}^n:= (\tilde{\Omega} ,\tilde{\mathcal{F}}, \{\tilde{\mathcal{F}}_{t}^n\} _{t\ge 0},\tilde{\mathbb{P}}  )$. Furthermore,  by \cite[ Lemma 5.7]{fischer2018existence}, the process $\tilde{W}_n$ and $\tilde{W}$ are $Q$-Wiener processes adapted to the filtration $\{\tilde{\mathcal{F}}_{t}^n\} _{t\ge 0} $ and $\{\tilde{\mathcal{F}}_{t}\} _{t\ge 0} $, respectively. They can be written as \[
\tilde{W}_n(t)=\sum_{l=1}^{\infty}\sqrt{\lambda_l}e_l\tilde{W}_l^n(t),~\tilde{W}(t)=\sum_{l=1}^{\infty}\sqrt{\lambda_l}e_l\tilde{W}_l(t),
\] 
where $\{\tilde W_l^n\}_{l\in\mathbb N}$ and $\{\tilde W_l\}_{l\in\mathbb N}$ are independent standard
	one-dimensional Brownian motions on $\tilde{\mathfrak A}^n$ and $\tilde{\mathfrak A}$, respectively.
    
    The space   $C([0,T];H_n)$ is a Borel subset of
	$Z_T$ and  since $u_n\in C([0,T];H_n)$ $\mathbb P$-a.s. and  $\tilde{u}_n$ and ${u}_n$ have the same law on $Z_T$,  we have   
        \[
        \mathcal{L}(\tilde{u}_n)\big(C([0,T];H_n)\big)=1,\quad\forall n\ge1.
        \]

        Since $L_{+}^2(\mathcal{O}):=\{v\in L^2(\mathcal{O}):v(x)\ge 0\, a.e.\,\text{in}\,\mathcal{O}\}$ is a closed subset  in $L_w^2(\mathcal{O})$,  then  \[A:=\big\{v\in C([0,T];L_w^2(\mathcal{O})):v(t)\in L_{+}^2(\mathcal{O}),\,\forall t\in[0,T]\big\}\in \mathcal B(Z_T).\] 
    	As $\tilde{u}_n$ and ${u}_n$ have the same law on $Z_T$, we have 
	\[
	\mathbb P\big(u_n(t)\ge0,\,\forall t\in[0,T]\big)=\tilde{\mathbb P}\big(\tilde u_n(t)\ge0,\,\forall t\in[0,T]\big)=1.
	\]
    Furthermore, $\tilde{u}_n\to \tilde{u}$ in $C([0,T];L_w^2(\mathcal{O}))\,\mathbb{P}$-a.s. yields
    \[
	\tilde{\mathbb P}\Big(\tilde u(t)\ge0,\ \forall t\in[0,T]\Big)=1.
	\]  
    Since  $\tilde{u}_n$ and ${u}_n$ have the same law on $Z_T$, under the same assumptions of Theorem~\ref{thm1}, there exist constants $C(T)$ and $C(T,k)$ such that 
        \begin{align}\label{a12}
\sup_{n\ge1}\tilde{\mathbb{E}}\Big(\sup_{t\in[0,T]}\|\tilde{u}_n(t)\|_{L^2(\mathcal{O})}^2\Big)\leq C(T),\\\label{a13}
\sup_{n\ge1}\tilde{\mathbb{E}}\Big(\int_{0}^T\|\nabla \tilde{u}_n(t)\|_{L^2(\mathcal{O})}^2dt\Big)\leq C(T),\\\label{c1}
\sup_{n\ge1}\tilde{\mathbb{E}}\Big(\int_{0}^T\|\tilde{u}_n(t)\|_{L^k(\mathcal{O})}^kdt\Big)\leq C(T,k),
    \end{align}
    where
	$k\in\big(2,\frac{2}{d+2}(\alpha-1+\beta)+2\big)$.
    By $\tilde u_n\to\tilde u$ in $L^k(0,T;L^k(\mathcal O))$ $\tilde{\mathbb P}\text{-a.s.}$, we use  Fatou's lemma and \eqref{c1} to obtain 
	\begin{align}\label{kkk}
		\tilde{\mathbb E}\Big(\int_0^T\|\tilde u(t)\|_{L^k(\mathcal O)}^kdt\Big)
		\le \liminf_{n\to\infty}\tilde{\mathbb E}\Big(\int_0^T\|\tilde u_n(t)\|_{L^k(\mathcal O)}^kdt\Big)
		\le C(T,k).
	\end{align}
	Similarly, by \eqref{a12},\eqref{a13} and Fatou's lemma, we have 
	\begin{align}\label{e3}
		\tilde{\mathbb E}\Big(\sup_{t\in[0,T]}\|\tilde u(t)\|_{L^2(\mathcal O)}^2\Big)
		\le C(T),\\\label{e4}
		\tilde{\mathbb{E}}\Big(\int_{0}^T\|\nabla \tilde{u}(t)\|_{L^2(\mathcal{O})}^2dt\Big)
		\le C(T).
	\end{align}
\subsection{The proof of the Theorem \ref{thm1}}
 Based on the previous analysis of the  Galerkin approximate equation (\ref{equ2}) and tightness criterion, we will show that the limit $(\tilde{u},\tilde{W})$ is satisfied with  equation (\ref{equ1}).
    Let $\phi\in H^s(\mathcal{O}),\,s>\frac{d}{2}+1$ satisfying $\nabla\phi\cdot\nu=0$ on $\partial\mathcal{O}$, for $t\in[0,T]$, we define 
 \begin{align*}
     &\Lambda _n (\tilde{u}_n,\tilde{W}_n,\phi)(t)\\
     :=&(\tilde{u}_n(0),\phi)_{L^2(\mathcal{O})}+\int _0^t\Big \langle\Pi_n\Delta\tilde{u}_n(s),\phi \Big \rangle ds-  \int _0^t\Big \langle\Pi_n\nabla\cdot\big(\tilde{u}_n(s)\nabla \tilde{c}_n(s)\big),  \phi\Big \rangle ds
   \\& +\int _0^t\Big \langle \Pi_n\tilde{u}_n^\alpha(s)\big(1-\int_\mathcal{O}\tilde{u}_n^{\beta}(s)dx\big),\phi \Big \rangle ds+\Big(\int _0^t\Pi_n\sigma(\tilde{u}_n(s))d\tilde{W}_n(s),\phi \Big)_{L^2(\mathcal{O})}\\=&(\tilde{u}_n(0),\phi)_{L^2(\mathcal{O})}+\int _0^t\Big \langle\Delta\tilde{u}_n(s),\phi \Big \rangle ds+  \int _0^t\Big \langle\tilde{u}_n(s)\nabla \tilde{c}_n(s),  \nabla\Pi_n\phi\Big \rangle ds
   \\& +\int _0^t\Big \langle \tilde{u}_n^\alpha(s)\big(1-\int_\mathcal{O}\tilde{u}_n^{\beta}(s)dx\big),\Pi_n\phi \Big \rangle ds+\Big(\int _0^t\Pi_n\sigma(\tilde{u}_n(s))d\tilde{W}_n(s),\phi \Big)_{L^2(\mathcal{O})}
 \end{align*}
 \begin{align*}
     \Lambda (\tilde{u},\tilde{W},\phi)(t):=&(\tilde{u}(0),\phi)_{L^2(\mathcal{O})}+\int _0^t\Big \langle\Delta\tilde{u}(s),\phi \Big \rangle ds+  \int _0^t\Big \langle\tilde{u}(s)\nabla \tilde{c}(s),\nabla \phi\Big \rangle ds
    \\&+\int _0^t\Big \langle \tilde{u}^\alpha(s)\big(1-\int_\mathcal{O}\tilde{u}^{\beta}(s)dx\big),\phi \Big \rangle ds+\Big(\int _0^t\sigma(\tilde{u}(s))d\tilde{W}(s),\phi \Big)_{L^2(\mathcal{O})}.
 \end{align*}

Next, we focus on proving the following two convergences.
 \begin{lem} \label{lem4.2}
    It holds that
    \begin{align}
    &\label{lim_1}\lim_{n\to\infty}\int_0^T\tilde{\mathbb{E}}\Big| (\tilde{u}_n(t),\phi)_{L^2(\mathcal{O})}-(\tilde{u}(t),\phi)\Big|dt=0,\\\label{lim_2}
&\lim_{n\to\infty}\int_0^T\tilde{\mathbb{E}}\Big|\Lambda _n (\tilde{u}_n,\tilde{W}_n,\phi)(t)-\Lambda (\tilde{u},\tilde{W},\phi)(t)\Big|dt=0.
     \end{align}
 \end{lem}
 \begin{proof} 
 By the convergence
$\tilde u_n\to \tilde u$ in $Z_T$ and  the uniform estimates \eqref{a12}-\eqref{e4},
we will apply Vitali's convergence theorem to prove. This approach is motivated by  \cite[Lemma 16]{dhariwal2019global} and \cite[Lemma 2.1]{debussche2011local}.
     For (\ref{lim_1}). Since
     $\tilde{u}_n\to\tilde{u}~\text{ in}~C([0,T];L_w^2(\mathcal{O}))\,\tilde{\mathbb{P}}\,\text{-a.s.}, 
     $  it holds that for any $t\in[0,T]$,
     \begin{align*}
         \lim_{n\to\infty}(\tilde{u}_n(t),\phi)_{L^2(\mathcal{O})}=(\tilde{u}(t),\phi)_{L^2(\mathcal{O})}\,\tilde{\mathbb P}\text{-a.s.}
     \end{align*}
     	Then, by Cauchy-Schwarz and \eqref{a12},\eqref{e3}, we have
        \begin{align*}
            &\sup_{n\ge1}\tilde{\mathbb E}\big|(\tilde u_n(t)-\tilde u(t),\phi)_{L^2(\mathcal{O})}\big|^2
		\\\le&C \|\phi\|_{L^2(\mathcal{O})}^2 \sup_{n\ge1}\tilde{\mathbb E}\big(\|\tilde u_n(t)\|_{L^2(\mathcal {O})}^2+\|\tilde u(t)\|_{L^2(\mathcal {O})}^2\big)
		\le C(T)\|\phi\|_{L^2(\mathcal{O})}^2.
        \end{align*}
		Hence, $(\tilde u_n(t)-\tilde u(t),\phi)_{L^2(\mathcal{O})}$ is uniformly integrable. Then, Vitali's theorem yields \[\tilde{\mathbb E}\big|(\tilde u_n(t)-\tilde u(t),\phi)_{L^2(\mathcal{O})}\big|\to0,\quad n\to \infty.\]
        Using the dominated convergence theorem in time integration yields \eqref{lim_1} holds.
        
        Next, we use the same way to verify \eqref{lim_2} holds. 
     \begin{align*}
         &\int_0^T\tilde{\mathbb{E}}\Big|\Lambda _n (\tilde{u}_n,\tilde{W}_n,\phi,\phi)(t)-\Lambda (\tilde{u},\tilde{W},\phi,\phi)(t)\Big|dt\\
    \leq&T\tilde{\mathbb{E}}\Big|\big(\tilde{u}_n(0)-\tilde{u}(0),\phi\big)_{L^2(\mathcal{O})}\Big|+ \tilde{\mathbb{E}}\int_0^T\Big|\int_0^t\Big \langle\nabla\tilde{u}_n(s)-\nabla\tilde{u}(s),\nabla\phi\Big \rangle ds\Big|dt\\&+\tilde{\mathbb{E}}\int_0^T\Big| \int_0^t\Big \langle\tilde{u}_n(s)\nabla\tilde{c}_n(s)-\tilde{u}(s)\nabla\tilde{c}(s),\nabla\phi\Big \rangle ds\Big|dt\\&+\tilde{\mathbb{E}}\int_0^T\Big| \int_0^t\Big \langle\tilde{u}_n(s)\nabla\tilde{c}_n(s),\nabla\Pi_n\phi-\nabla\phi\Big \rangle ds\Big|dt\\&+\tilde{\mathbb{E}}\int_0^T\Big|\int_0^t\left \langle\tilde{u}_n^\alpha(s)\big(1-\int_\mathcal{O}\tilde{u}_n^{\beta}(s)dx\big)-\tilde{u}^\alpha(s)\big(1-\int_\mathcal{O}\tilde{u}^{\beta}(s)dx\big),\phi\right \rangle ds\Big|dt\\&+\tilde{\mathbb{E}}\int_0^T\Big|\int_0^t\left \langle\tilde{u}_n^\alpha(s)\big(1-\int_\mathcal{O}\tilde{u}_n^{\beta}(s)dx\big),\Pi_n\phi-\phi\right \rangle ds\Big|dt\\&+\tilde{\mathbb{E}}\int_0^T\Big|\int_0^t\Big (\Pi_n\sigma(\tilde{u}_n(s))d\tilde{W}_n(s)-\sigma(\tilde{u}(s))d\tilde{W}(s),\phi\Big)_{L^2(\mathcal{O})}\Big|dt\\&:=\sum_{i=1}^{7}\tilde{\mathbb{E}}\int_0^T\big|I_i(t)\big|dt.
     \end{align*} It remains to show each term separately.  By \eqref{lim_1}, it is easy to obtain $\tilde{\mathbb{E}}\displaystyle\int_0^T\big|I_1(t)\big|dt\to0$.\\
        For $\tilde{\mathbb{E}}\displaystyle\int_0^T\big|I_2(t)\big|dt$.
     Since $\tilde{u}_n\to\tilde{u}~\text{weakly~ in}~L^2(0,T;H^1(\mathcal{O}))~\tilde{\mathbb{P}}\text{-a.s.} ,$
     which means that for every $t\in[0,T]$ and $\phi\in H^s(\mathcal{O})$, it holds
     \begin{align}\label{a22}
          \lim_{n\to\infty}\int_0^t\Big \langle \nabla\tilde{u}_n(s)-\nabla\tilde{u}(s),\nabla\phi\Big \rangle ds=0,~ \tilde{\mathbb{P}}\text{-a.s.}
     \end{align}
    Using the H\"older inequality in time and (\ref{a13}),  we obtain  
    \begin{align*}
        \sup_{n\ge 1}\tilde{\mathbb{E}}\Big(\Big|\int_0^t\Big \langle\nabla\tilde{u}_n(s),\nabla\phi\Big \rangle ds\Big|^2\Big)&\leq \|\nabla\phi\|_{L^2(\mathcal{O})}^2\sup_{n\ge 1}\tilde{\mathbb{E}}\Big(\Big|\int_0^t\|\nabla\tilde{u}_n(s)\|_{L^2(\mathcal{O})} ds\Big|^2\Big)\\&\leq T\sup_{n\ge 1}\tilde{\mathbb{E}}\Big(\int_0^t\|\nabla\tilde{u}_n(s)\|_{L^2(\mathcal{O})}^2 ds\Big)\leq C(T).
    \end{align*}
    Similarly, by (\ref{e4}), we know $\sup_{n\ge 1}\tilde{\mathbb{E}}\Big(\Big|\displaystyle\int_0^t\big \langle\nabla\tilde{u}(s),\nabla\phi\big \rangle ds\Big|^2\Big)\leq C(T).$
    Thus, $|I_2(t)|$ is uniformly integrable.   Applying Vitali's theorem  infers  
    \begin{align*}
        \tilde{\mathbb{E}}\Big|\int_0^t\Big \langle\nabla\tilde{u}_n(s)-\nabla\tilde{u}(s),\nabla\phi\Big \rangle ds\Big|\to 0,~n\to \infty.
    \end{align*}
    Furthermore, we use the dominated convergence theorem in time integration to get  \[\lim_{n\to \infty}\tilde{\mathbb{E}}\int_0^T\big|I_2(t)\big|dt=0.\] 
    For $\tilde{\mathbb{E}}\displaystyle\int_0^T\big|I_3(t)\big|dt$. By the H\"older inequality, (\ref{c}) and the continuous embedding $H^s(\mathcal{O})\hookrightarrow W^{1,\infty}(\mathcal{O})$, we have
    \begin{align*}
     &\Big|\int_0^t\Big \langle\tilde{u}_n(s)\nabla\tilde{c}_n(s)-\tilde{u}(s)\nabla\tilde{c}(s),\nabla\phi\Big \rangle ds\Big|\\
     \leq& \Big|\int_0^t\Big \langle\big(\tilde{u}_n(s)-\tilde{u}(s)\big)\nabla\tilde{c}_n(s)+\tilde{u}(s)\big(\nabla\tilde{c}_n(s)-\nabla\tilde{c}(s)\big),\nabla\phi\Big \rangle ds\Big|\\
     \leq& \|\nabla\phi\|_{L^\infty(\mathcal{O})}\int_0^t \|\tilde{u}_n(s)-\tilde{u}(s)\|_{L^2(\mathcal{O})}\|\nabla\tilde{c}_n(s)\|_{L^2(\mathcal{O})}ds\\&+\|\nabla\phi\|_{L^\infty(\mathcal{O})}\int_0^t \|\tilde{u}(s)\|_{L^2(\mathcal{O})}\|\nabla\tilde{c}_n(s)-\nabla\tilde{c}(s)\|_{L^2(\mathcal{O})}ds\\
     \leq &C \|\phi\|_{H^s(\mathcal{O})}\|\tilde{u}_n-\tilde{u}\|_{L^2(0,T;L^2(\mathcal{O}))}\Big(\|\tilde{u}_n\|_{L^2(0,T;L^2(\mathcal{O}))}+\|\tilde{u}\|_{L^2(0,T;L^2(\mathcal{O}))}\Big).
    \end{align*}
   Since $\tilde{u}_n\to\tilde{u}~\text{ in}~L^2(0,T;L^2(\mathcal{O}))~\tilde{\mathbb{P}}\text{-a.s.}$, it follows that 
     \begin{align}
         \lim_{n\to\infty}\int_0^t\Big \langle(\tilde{u}_n(s)\nabla\tilde{c}_n(s))-\tilde{u}(s)\nabla\tilde{c}(s),\nabla\phi\Big \rangle ds=0,~\tilde{\mathbb{P}}\text{-a.s.}
     \end{align}
      Applying $L^{\alpha+1}(\mathcal{O})\hookrightarrow L^2(\mathcal{O})$,  the H\"older inequality and (\ref{c1}) to obtain
     \begin{align}\label{4.21}
\nonumber&\sup_{n\ge1}\nonumber\tilde{\mathbb{E}}\Big(\Big|\int_0^t\Big \langle\tilde{u}_n(s)\nabla\tilde{c}_n(s),\nabla\phi\Big \rangle ds\Big|^\frac{\alpha+1}{2}\Big)\\\nonumber\leq& \|\nabla\phi\|_{L^\infty(\mathcal{O})}^{\frac{\alpha+1}{2}}\sup_{n\ge1}\tilde{\mathbb{E}}\Big(\Big|\int_0^t\|\tilde{u}_n(s)\|_{L^2(\mathcal{O})}^2ds\Big|^\frac{\alpha+1}{2} \Big)\\\leq& C T^\frac{\alpha-1}{2} \sup_{n\ge1}\tilde{\mathbb{E}}\Big(\Big|\int_0^t\|\tilde{u}_n(s)\|_{L^{\alpha+1}(\mathcal{O})}^{\alpha+1}ds\Big)\leq C(T,\alpha).
    \end{align}
    Similarly, by  (\ref{kkk}), it holds $\tilde{\mathbb{E}}\Big(\Big|\displaystyle\int_0^t\big\langle\tilde{u}(s)\nabla\tilde{c}(s),\nabla\phi \big\rangle ds\Big|^\frac{\alpha+1}{2}\Big)\leq C$. Thus  $|I_2(t)|$ is uniformly integrable. By Vitali's theorem and the dominated convergence theorem, we  completes the proof of $\lim_{n\to \infty}\tilde{\mathbb{E}}\displaystyle\int_0^T\big|I_3(t)\big|dt=0$. Let $\frac{d}{2}+1< \tilde{s}<s$. Using the interpolation inequality,
    it holds 
    \begin{align}\label{f6}
       \nonumber& \lim_{n\to\infty}\|\nabla(\Pi_n\phi-\phi)\|_{L^\infty(\mathcal{O})}\leq C\lim_{n\to\infty}\|\Pi_n\phi-\phi\|_{H^{\tilde{s}}(\mathcal{O})}\\\leq& C\lim_{n\to\infty}\|\Pi_n\phi-\phi\|_{L^2(\mathcal{O})}^{1-\frac{\tilde{s}}{s}}\big(\|\Pi_n\phi\|_{{H^s(\mathcal{O})}}^{\frac{\tilde{s}}{s}}+\|\phi\|_{{H^s(\mathcal{O})}}^{\frac{\tilde{s}}{s}}\big)=0.
    \end{align}
    Hence, combining with \eqref{4.21}, we have $\lim_{n\to \infty}\tilde{\mathbb{E}}\displaystyle\int_0^T\big|I_4(t)\big|dt=0$.\\
For $\tilde{\mathbb{E}}\displaystyle\int_0^T\big|I_5(t)\big|dt$. We have 
    \begin{align}\label{f1}
         \nonumber&\Big|\int_0^t\left \langle\big(1-\int_\mathcal{O}\tilde{u}_n^{\beta}(s)dx\big)\tilde{u}_n^\alpha(s)-\big(1-\int_\mathcal{O}\tilde{u}^{\beta}(s)dx\big)\tilde{u}^\alpha(s),\phi\right \rangle ds\Big|\\\nonumber
         \leq &\Big|\int_0^t\left \langle \big(\tilde{u}_n^\alpha(s)-\tilde{u}^\alpha(s)\big)\big(1-\int_\mathcal{O}\tilde{u}_n^{\beta}(s)dx\big),\phi\right \rangle ds\Big|\\\nonumber
         &+\Big|\int_0^t\left \langle\tilde{u}_n^\alpha(s)\big(\int_\mathcal{O}\tilde{u}_n^{\beta}(s)dx-\int_\mathcal{O}\tilde{u}^{\beta}(s)dx\big),\phi\right \rangle ds\Big|\\\nonumber
         \leq& \|\phi\|_{L^\infty(\mathcal{O})}\int_0^t\|\tilde{u}^\alpha_n(s)-\tilde{u}^\alpha(s)\|_{L^1(\mathcal{O})}\big(1+\|\tilde{u}_n(s)\|_{L^{\beta}(\mathcal{O})}^{\beta}\big)ds\\&+\|\phi\|_{L^\infty(\mathcal{O})}\int_0^t\|\tilde{u}_n(s)\|_{L^\alpha(\mathcal{O})}^\alpha\|\tilde{u}^\beta_n(s)-\tilde{u}^\beta(s)\|_{L^{1}(\mathcal{O})}ds
    \end{align}
    By  the mean value theorem  and the H\"older inequality,  it holds
    \begin{align*}
        &\|\tilde u_n^\alpha(s)-\tilde u^\alpha(s)\|_{L^1(\mathcal O)}
		\\\le& \alpha \int_{\mathcal O}|\tilde u_n^{\alpha-1}(s)+\tilde u^{\alpha-1}(s)|
		|\tilde u_n(s)-\tilde u(s)|dx \\\leq&\alpha \big(\|\tilde u_n(s)\|_{L^{\alpha+\beta}(\mathcal O)}^{\alpha-1}+\|\tilde u(s)\|_{L^{\alpha+\beta}(\mathcal O)}^{\alpha-1}\big)
		\|\tilde u_n(s)-\tilde u(s)\|_{L^{\alpha+\beta}(\mathcal O)}
		\|1\|_{L^{\frac{{\alpha+\beta}}{\beta}}(\mathcal O)}\\\leq &C(\alpha,\beta)\big(\|\tilde u_n(s)\|_{L^{\alpha+\beta}(\mathcal O)}^{\alpha-1}+\|\tilde u(s)\|_{L^{\alpha+\beta}(\mathcal O)}^{\alpha-1}\big)
		\|\tilde u_n(s)-\tilde u(s)\|_{L^{\alpha+\beta}(\mathcal O)}.
    \end{align*} 
     Similarly, 
    \begin{align*}
        \|\tilde u_n^\beta(s)-\tilde u^\beta(s)\|_{L^1(\mathcal O)}\leq C(\alpha,\beta)\big(\|\tilde u_n(s)\|_{L^{\alpha+\beta}(\mathcal O)}^{\beta-1}+\|\tilde u(s)\|_{L^{\alpha+\beta}(\mathcal O)}^{\beta-1}\big)
		\|\tilde u_n(s)-\tilde u(s)\|_{L^{\alpha+\beta}(\mathcal O)}.
    \end{align*}
    Then, a typical term of \eqref{f1} can be estimated by using the H\"older inequality, 
    \begin{align}\label{f2}
        \nonumber&\int_0^t\|\tilde u(s)\|_{L^{\alpha+\beta}(\mathcal O)}^{\alpha-1}
		\|\tilde u_n(s)-\tilde u(s)\|_{L^{\alpha+\beta}(\mathcal O)}\|\tilde{u}_n(s)\|_{L^{\beta}(\mathcal{O})}^{\beta}ds\\\nonumber\leq& C\Big(\int_0^t\|\tilde u(s)\|_{L^{\alpha+\beta}(\mathcal O)}^{{\alpha+\beta}}ds\Big)^\frac{\alpha-1}{{\alpha+\beta}}
		\Big(\int_0^t\|\tilde u_n(s)-\tilde u(s)\|_{L^{\alpha+\beta}(\mathcal O)}^{\alpha+\beta}ds\Big)^\frac{1}{{\alpha+\beta}}\\&\cdot	\Big(\int_0^t\|\tilde{u}_n(s)\|_{L^{{\alpha+\beta}}(\mathcal{O})}^{{\alpha+\beta}}ds\Big)^\frac{\beta}{{\alpha+\beta}}.
    \end{align}
    Furthermore, the rest terms of \eqref{f1} can be estimated similarly with (\ref{f2}). Since  $\tilde u_n\to\tilde u$ in $L^{\alpha+\beta}(0,T;L^{\alpha+\beta}(\mathcal O))~\tilde{\mathbb{P}}\,\text{-a.s.}$ it follows that 
    \begin{align*}
        \lim_{n\to\infty}\int_0^t\left \langle\big(1-\int_\mathcal{O}\tilde{u}_n^{\beta}(s)dx\big)\tilde{u}_n^\alpha(s)-\big(1-\int_\mathcal{O}\tilde{u}^{\beta}(s)dx\big)\tilde{u}^\alpha(s),\phi\right \rangle ds=0,~\tilde{\mathbb{P}}\text{-a.s.}
    \end{align*}
   Taking $r\in(\alpha+\beta,\frac{2}{d+2}(\alpha-1+\beta)+2)$, by inequalities (\ref{c1})  and the H\"older and the Jensen inequality,  we have the following 
     \begin{align}\label{f7}
        \nonumber&\sup_{n\ge 1}\tilde{\mathbb{E}}\Big(\Big|\int_0^t\Big \langle\tilde{u}_n^\alpha(s)(1-\int_\mathcal{O}\tilde{u}_n^{\beta}(s)dx),\phi\Big \rangle ds\Big|^\frac{r}{\alpha+\beta}\Big)\\\nonumber
        \leq& C\|\phi\|_{L^\infty(\mathcal{O})}^\frac{r}{\alpha+\beta}\sup_{n\ge 1}\tilde{\mathbb{E}}\Big(\Big|\int_0^t\big(\|\tilde{u}_n(s)\|_{L^\alpha(\mathcal{O})}^\alpha+\|\tilde{u}_n(s)\|_{L^r(\mathcal{O})}^{\alpha+\beta} \big)ds\Big|^\frac{r}{\alpha+\beta}\Big)\\\nonumber
        \leq& C T^\frac{r-\alpha}{\alpha+\beta}\|\phi\|_{L^\infty(\mathcal{O})}^\frac{r}{\alpha+\beta} \sup_{n\ge 1}\tilde{\mathbb{E}}\Big[\Big(\int_0^t\|\tilde{u}_n(s)\|_{L^{r}(\mathcal{O})}^{r}ds\Big)^\frac{\alpha}{\alpha+\beta}\Big]\\\nonumber
        &+C  T^\frac{r-\alpha-\beta}{\alpha+\beta}\|\phi\|_{L^\infty(\mathcal{O})}^\frac{r}{\alpha+\beta}\sup_{n\ge 1}\tilde{\mathbb{E}}\Big(\int_0^t\|\tilde{u}_n(s)\|_{L^{r}(\mathcal{O})}^r ds\Big)\\\nonumber
       \leq &C T^\frac{r-\alpha}{\alpha+\beta}\|\phi\|_{L^\infty(\mathcal{O})}^\frac{r}{\alpha+\beta} \Big(\sup_{n\ge 1}\tilde{\mathbb{E}}\int_0^t\|\tilde{u}_n(s)\|_{L^{r}(\mathcal{O})}^{r}ds\Big)^\frac{\alpha}{\alpha+\beta}\\
        &+C  T^\frac{r-\alpha-\beta}{\alpha+\beta}\|\phi\|_{L^\infty(\mathcal{O})}^\frac{r}{\alpha+\beta}\sup_{n\ge 1}\tilde{\mathbb{E}}\Big(\int_0^t\|\tilde{u}_n(s)\|_{L^{r}(\mathcal{O})}^r ds\Big)
        \leq C(T,\alpha,\beta).
    \end{align}
   Therefore,  $|I_5(t)|$ is uniformly integrable. We apply Vitali's theorem and the dominated convergence theorem to infer that $\lim_{n\to\infty}\tilde{\mathbb{E}}\displaystyle\int_0^T\big|I_5(t)\big|dt=0$ holds. Furthermore, combining with \eqref{f6} and (\ref{f7}), we get $\lim_{n\to\infty}\tilde{\mathbb{E}}\displaystyle\int_0^T\big|I_6(t)\big|dt=0$.\\ Finally, we prove $\lim_{n\to\infty}\tilde{\mathbb{E}}\displaystyle\int_0^T\big|I_7(t)\big|dt=0$.  By the properties of orthogonal projection \eqref{3.1}, \eqref{3.2} and  Assumption  \ref{ass}, we obtain
   \begin{align*}
		&\int_0^t \|\Pi_n\sigma(\tilde u_n(s))-\sigma(\tilde u(s))\|_{L_2(U_0,L^2(\mathcal{O}))}^2ds \\
		=& \int_0^t \big\|\Pi_n(\sigma(\tilde u_n(s))-\sigma(\tilde u(s))) + (I-\Pi_n)\sigma(\tilde u(s))\big\|_{L_2(U_0,L^2(\mathcal{O}))}^2\,ds \\
		\le& 2\int_0^t \|\Pi_n(\sigma(\tilde u_n(s))-\sigma(\tilde u(s)))\|_{L_2(U_0,L^2(\mathcal{O}))}^2ds
		+2\int_0^t \|(\Pi_n-I)\sigma(\tilde u(s))\|_{L_2(U_0,L^2(\mathcal {O}))}^2ds \\
		\le &2\int_0^t \|\sigma(\tilde u_n(s))-\sigma(\tilde u(s))\|_{L_2(U_0,L^2(\mathcal{O}))}^2\,ds
		+2\int_0^t \|(I-\Pi_n)\sigma(\tilde u(s))\|_{L_2(U_0,L^2(\mathcal{O}))}^2ds\\\le&C\int_0^t\big(\|\tilde{u}_n(s)\|_{L^{\gamma}(\mathcal{O})}^{\gamma-2}+\|\tilde{u}(s)\|_{L^{\gamma}(\mathcal{O})}^{\gamma-2}\big)\|\tilde{u}_n(s)-\tilde{u}(s)\|_{L^\gamma(\mathcal{O})}^2 ds\\&+2\int_0^t \|(I-\Pi_n)\sigma(\tilde u(s))\|_{L_2(U_0,L^2(\mathcal{O}))}^2ds\\
      \leq &C\Big(\int_0^t\|\tilde{u}_n(s)-\tilde{u}(s)\|_{L^\gamma(\mathcal{O})}^\gamma ds\Big)^\frac{2}{\gamma}\Big(\int_0^t\big(\|\tilde{u}_n(s)\|_{L^{\gamma}(\mathcal{O})}^{\gamma}+\|\tilde{u}(s)\|_{L^{\gamma}(\mathcal{O})}^{\gamma}\big)ds\Big)^\frac{\gamma-2}{\gamma}\\&+2\int_0^t \|(I-\Pi_n)\sigma(\tilde u(s))\|_{L_2(U_0,L^2(\mathcal{O}))}^2ds .
	\end{align*}
   For each fixed $(\omega,s)$, it holds 
	\[
	\|(I-\Pi_n)\sigma(\tilde u(s))\|_{L_2(U_0,L^2(\mathcal {O}))}\le\|I-\Pi_n\|_{L(L^2(\mathcal{O}))}\|\sigma(\tilde u(s))\|_{L_2(U_0,L^2(\mathcal {O}))}\to0,\quad n\to\infty.
	\]
		Moreover, since 
	$\tilde{\mathbb E}\Big(\displaystyle\int_0^T\|\sigma(\tilde u(s))\|_{L_2(U_0,L^2(\mathcal {O}))}^2ds\Big)<\tilde{\mathbb E}\Big(\displaystyle\int_0^T\|\tilde u(s)\|_{L^\gamma(\mathcal {O}))}^\gamma ds\Big)$,
	using the dominated convergence yields
	\[
\tilde{\mathbb E}\Big(\int_0^t\|(I-\Pi_n)\sigma(\tilde u(s))\|_{L_2}^2ds\Big)\to0.
\]
    Since  $\tilde u_n\to\tilde u$ in $L^{\gamma}(0,T;L^{\gamma}(\mathcal O))~\tilde{\mathbb{P}}\text{-a.s.}$ it shows \[\Pi_n\sigma(\tilde{u}_n)\to\sigma(\tilde{u}),\quad \text{in}\, L^2(0,T;L_2(U_0,L^2(\mathcal{O})))\,\tilde{\mathbb{P}}\text{-a.s.}\] Furthermore, by \cite[Lemma 2.1]{debussche2011local}, since $\tilde{W}_n\to\tilde{W}$ in $C([0,T];U_0)$,   it implies that 
     \begin{align}\label{c5}
\int_0^t\Pi_n\sigma(\tilde{u}_n(s))d\tilde{W}_n(s)\to\int_0^t\sigma(\tilde{u}(s))d\tilde{W}(s)\quad\text{in probability in }L^2(0,T;L^2(\mathcal{O})).
     \end{align}
     Then, we use the Burkholder-Davis-Gundy inequality, Fubini theorem  and (\ref{kkk}) to infer that
     \begin{align}\label{c6}
         \nonumber&\sup_{n\ge 1}\tilde{\mathbb{E}}\Big\|\int_0^t\Pi_n\sigma(\tilde{u}_n(s))d\tilde{W}_n(s)-\int_0^t\sigma(\tilde{u}(s))d\tilde{W}(s)\Big\|_{L^2(0,T;L^2(\mathcal{O}))}^2\\\nonumber
         \leq&C\sup_{n\ge 1}\tilde{\mathbb{E}}\int_0^T\Big (\int_0^t\big(\|\sigma(\tilde{u}_n(s))\|_{L_2(U_0,L^2(\mathcal{O}))}^2 +\|\sigma(\tilde{u}(s))\|_{L_2(U_0,L^2(\mathcal{O}))}^2 \big)ds\Big)dt\\
         \leq &C (T)\sup_{n\ge 1}\tilde{\mathbb{E}}\Big (\int_0^T\big(\|\tilde{u}_n(s)\|_{L^\gamma(\mathcal{O})}^\gamma+\|\tilde{u}(s)\|_{L^\gamma(\mathcal{O})}^\gamma\big) ds\Big)
         \leq C(T).
     \end{align}
     Combining with (\ref{c5}), (\ref{c6}) and using Vitali's theorem, we infer that 
     \begin{align*}
\nonumber&\tilde{\mathbb{E}}\Big\|\int_0^t\Pi_n\sigma(\tilde{u}_n(s))d\tilde{W}_n(s)-\int_0^t\sigma(\tilde{u}(s))d\tilde{W}(s)\Big\|_{L^2(0,T;L^2(\mathcal{O}))}\to 0, n\to \infty.
     \end{align*} Furthermore, by Cauchy-Schwarz and Jensen's inequality,
	\begin{align*}
		&\tilde{\mathbb{E}}\int_0^T\big|I_7(t)\big|dt
		\le \|\phi\|_{L^2(\mathcal O)}\tilde{\mathbb{E}}\Big(\int_0^T
		\Big\|\int_0^t\Pi_n\sigma(\tilde{u}_n(s))d\tilde{W}_n(s)
		-\int_0^t\sigma(\tilde{u}(s))d\tilde{W}(s)\Big\|_{L^2(\mathcal O)}dt\Big)\\
		\le& \|\phi\|_{L^2(\mathcal O)}T^{1/2}
		\Bigg(\tilde{\mathbb{E}}\int_0^T
		\Big\|\int_0^t\Pi_n\sigma(\tilde{u}_n(s))d\tilde{W}_n(s)
		-\int_0^t\sigma(\tilde{u}(s))d\tilde{W}(s)\Big\|_{L^2(\mathcal O)}^2dt\Bigg)^{1/2}.
	\end{align*}
	Hence, $\lim_{n\to\infty}\tilde{\mathbb{E}}\displaystyle\int_0^T\big|I_7(t)\big|dt=0$. Therefore, we complete the  proof of Lemma \ref{lem4.2}.
 \end{proof}
 \begin{proof}[The proof of Theorem \ref{thm1} ]
     Since $u_n$ is the solution to (\ref{equ2}) and $(u_n,W_n)$   with $n\in\mathbb{N}$ has the same law  as $(\tilde{u}_n,\tilde{W}_n)$, we have  
 \begin{align*}
     &\int_0^T{\mathbb{E}}\Big|({u}_n(t),\phi)_{L^2(\mathcal{O})}-\Lambda _n ({u}_n,{W},\phi)\Big|dt\\=&\int_0^T\tilde{\mathbb{E}}\Big|(\tilde{u}_n(t),\phi)_{L^2(\mathcal{O})}-\Lambda _n (\tilde{u}_n,\tilde{W}_n,\phi)\Big|dt=0.
 \end{align*}
 Furthermore, by Lemma \ref{lem4.2}, we infer that
 \begin{align*}
    \nonumber &\int_0^T\tilde{\mathbb{E}}\Big|(\tilde{u}(t),\phi)_{L^2(\mathcal{O})}-\Lambda (\tilde{u},\tilde{W},\phi)(t)\Big|dt\\ \nonumber 
     \leq& \int_0^T\tilde{\mathbb{E}}\Big|(\tilde{u}_n(t),\phi)_{L^2(\mathcal{O})}-(\tilde{u}(t),\phi)_{L^2(\mathcal{O})}\Big|dt\\&+\int_0^T\tilde{\mathbb{E}}\Big|(\tilde{u}_n(t),\phi)_{L^2(\mathcal{O})}-\Lambda _n (\tilde{u}_n,\tilde{W}_n,\phi)\Big|dt\\&+\int_0^T\tilde{\mathbb{E}}\Big|\Lambda _n (\tilde{u}_n,\tilde{W}_n,\phi)-\Lambda (\tilde{u},\tilde{W},\phi)\Big|dt=0
 \end{align*}
  Hence, it follows that for a.e.\ $t\in(0,T)$,
		\[
		(\tilde u(t),\phi)_{L^2(\mathcal O)}=\Lambda(\tilde u,\tilde W,\phi)(t)
		\qquad \tilde{\mathbb P}\text{-a.s.}
		\]
 Therefore,  for a.e. $t\in[0,T]$ and $\tilde{\mathbb{P}}$-a.s. $\omega\in\tilde{\Omega}$, we obtain
 \begin{align*}
    (\tilde{u}(t),\phi)_{L^2(\mathcal{O})}=&(\tilde{u}(0),\phi)_{L^2(\mathcal{O})}-\int _0^t\int_\mathcal{O}\nabla\tilde{u}(s)\cdot\nabla\phi dx ds+  \int _0^t\int_\mathcal{O} \tilde{u}(s)\nabla \tilde{c}(s)\cdot\nabla  \phi dxds
    \\&+\int _0^t\int_\mathcal{O} \tilde{u}^\alpha(s)(1-\int_\mathcal{O}\tilde{u}^{\beta}(s)dx)\phi dx ds+\Big(\int _0^t\sigma(\tilde{u}(s))d\tilde{W}(s),\phi \Big)_{L^2(\mathcal{O})}.
\end{align*} 
In particular, by a density argument, one can take the test functions $\phi\in H^1(\mathcal{O})$.
Therefore,  the stochastic equation (\ref{equ1}) has a martingale solution $\left(  \tilde{\mathfrak{A}}, \tilde{u},\tilde{W}\right)$. Moreover, since $\tilde{u}_n(0)\to \tilde{u}(0)$ in $L_w^2(\mathcal{O})$ $\tilde{\mathbb{P}}$-a.s., we have $\mathcal{L}(\tilde{u}_n(0))\Longrightarrow \mathcal{L}(\tilde{u}(0))$  on $L_w^2(\mathcal{O})$.  
    Consequently, due to the fact that $\tilde{u}_n(0)$ and ${u}_n(0)$ have the same law on $L_w^2(\mathcal{O})$ and $u_n(0)\to u_0$ in $L^2(\Omega;L^2(\mathcal{O}))$, we obtain that  $\tilde{u}(0)$ and ${u}_0$ have the same law on $L_w^2(\mathcal{O})$.
 \end{proof}
 \appendix \section{Appendix}
 In this appendix, we establish a compactness criterion on  a more general topological space $(\mathcal Z,\mathcal T)$, which will be used to verify tightness of the family of laws. 
Our argument is motivated by the compactness criterion in \cite[Section 2.1]{law2014existence}. Compared with \cite[Lemma 2.3]{law2014existence},   we introduce the space $L^{p}(0,T;E)$ together with a compact embedding $\mathbb{V}\hookrightarrow E$ and we do not require $\mathbb U$ and $\mathbb V$ to be Hilbert spaces. Instead, we allow $\mathbb U$ and $\mathbb V$ to be reflexive Banach spaces and only assume the continuous embeddings
	\[
	\mathbb U\hookrightarrow \mathbb H\cong \mathbb H'\hookrightarrow \mathbb U', 
	\qquad
	\mathbb{V}\hookrightarrow E\hookrightarrow \mathbb U'.
	\]
    This generality is more convenient for our  functional setting and approximation framework.
    We now state and prove the compactness criterion.
\begin{thm}
    [Compactness criterion]\label{lem4.5} 
    Let $\mathbb{H}$ be a real separable Hilbert space, $\mathbb{U}$ be a real reflexive Banach space such that the embedding  $\mathbb{U}\hookrightarrow \mathbb{H}$ is continuous and dense. Let  $\mathbb{U}'$ be the  dual space of $\mathbb{U}$   such that  the embedding $\mathbb{U}\hookrightarrow \mathbb{H}\cong \mathbb{H}'\hookrightarrow\mathbb{U}'$  are continuous.
         Let $E,\mathbb{V}$ be real reflexive  Banach spaces  such that the embedding $\mathbb{V}\hookrightarrow  E\hookrightarrow\mathbb{U}'$ are continuous and the embedding $\mathbb{V}\hookrightarrow E$ is compact.
          Let \begin{align}\label{Z}
    \mathcal{Z}
:= C([0,T];\mathbb U') \cap L^{q}_{w}(0,T;\mathbb{V}) \cap L^{p}(0,T;E) \cap C([0,T];\mathbb{H}_{w}),\,1<q\leq p<\infty,
\end{align} with the topology $\mathcal{T}$ induced by the supremum of  corresponding topologies. The definition of these functional spaces and the corresponding topologies can be found in \cite[Section 2.1]{law2014existence}. Then a set $\mathcal{K}\subset \mathcal{Z}$ is $\mathcal{T}$-relatively compact if the following
four conditions hold
\begin{itemize}
    \item[(a)] 
$\sup_{v\in \mathcal{K}}\sup_{s\in[0,T]}\|v(s)\|_{\mathbb{H}}<\infty,$
\item[(b)]
$\sup_{v\in \mathcal{K}}\int_{0}^{T}\|v(s)\|_{\mathbb{V}}^{q}ds<\infty,
\quad \text{i.e. } \mathcal{K} \text{ is bounded in } L^{q}(0,T;\mathbb{V}),$
\item[(c)]
$\sup_{v\in \mathcal{K}}\int_{0}^{T}\|v(s)\|_{E}^{r}ds<\infty,
\quad \text{i.e. } \mathcal{K} \text{ is bounded in } L^{r}(0,T;E)$, where $r>p$,
\item[(d)]
$\lim_{\delta\to 0}\sup_{v\in \mathcal{K}}\sup_{\substack{s,t\in[0,T]\\ |t-s|\le \delta}}
\|v(t)-v(s)\|_{\mathbb U'}=0$.
\end{itemize}
\end{thm}
\begin{proof}
	 Let $M:=\sup_{v\in\mathcal{K}}\sup_{t\in[0,T]}\|v(t)\|_{\mathbb H}<\infty$ and  $\mathbb B:=\{x\in\mathbb H:\|x\|_{\mathbb H}\le M\}$. Since $\mathbb{H}$ is separable, $C([0,T];\mathbb{B}_w)$ is  metrizable, see \cite{law2014existence}.  By (a) the set $\mathcal{K}\subset$ $C([0,T];\mathbb{B}_w)\subset C([0,T];\mathbb{H}_w)$.  Let $(v_n)$ be an arbitrary sequence in $\mathcal K$. We will prove that $(v_n)$ exists
	a subsequence converging in each corresponding topology of $\mathcal Z$.\\
    By (b), the set $\mathcal K$ is bounded in $L^q(0,T;\mathbb{V})$ and by (d), $\mathcal K$ is
	equicontinuous in $C([0,T];\mathbb U')$. Since the embeddings
	$
	\mathbb{V} \hookrightarrow E \hookrightarrow \mathbb U'
	$
	are continuous and $\mathbb{V}\hookrightarrow E$ is compact, and $\mathbb{V},E,\mathbb U'$
	are reflexive Banach spaces, we  apply Dubinsky's theorem
	\cite[Theorem 2.2]{law2014existence} to obtain $\mathcal K$ is
	relatively compact in $C([0,T];\mathbb U')$ and in $L^q(0,T;E)$.
	Therefore there exist a subsequence, still denoted by $(v_n)$ and 
	$v \in C([0,T];\mathbb U')\cap L^q(0,T;E)$ such that
	\begin{equation}\label{A.2}
		v_n \to v \ \text{ in } C([0,T];\mathbb U')
		\qquad\text{and}\qquad
		v_n \to v \, \text{ in } L^q(0,T;E).
	\end{equation}
Therefore, by interpolation inequality with $\theta\in(0,1)$ and $\frac{1}{p}=\frac{\theta}{q}+\frac{1-\theta}{r}$, we have
\begin{equation}
	\|v_n-v\|_{L^p(0,T;E)}=\|v_n-v\|^\theta_{L^q(0,T;E)}\|v_n-v\|^{1-\theta}_{L^{r}(0,T;E)}\rightarrow 0,
\end{equation} 
where the condition (c) is used.
By (b), the set $\mathcal K$ is bounded in $L^q(0,T;\mathbb V)$.
	Since $\mathbb V$ is reflexive  Banach space and   $1<q<\infty$, the  space
	$L^q(0,T;\mathbb V)$ is reflexive. Hence, by the Eberlein-\v{S}mulian theorem \cite[Theorem 3.19]{brezis2011functional}, passing to a further subsequence, still denoted by $(v_n)$,  we have
$v_n \rightharpoonup \bar v$ weakly in $L^q(0,T;\mathbb V)$. Since the embedding $\mathbb{V}\hookrightarrow\mathbb{U}'$ is continuous, the induced map
	$L^q(0,T;\mathbb V)\to L^q(0,T;\mathbb U')$ is continuous and linear. By \eqref{A.2} and  uniqueness of weak limits in $L^q(0,T;\mathbb U')$, we conclude that $\bar v=v$ in $L^q(0,T;\mathbb U')$. Consequently, $v\in L^q(0,T;\mathbb V)$ and
	\begin{align}\label{A.3}
	    v_n \to v \quad \text{in } L_w^q(0,T;\mathbb V).
	\end{align}
     By (a), we have
	$\sup_n\sup_{t}\|v_n(t)\|_{\mathbb H}\le r$. Since $v_n \to v \ \text{ in } C([0,T];\mathbb U')$ and the embedding $\mathbb U\hookrightarrow\mathbb H$ is continuous
	and dense, applying  \cite[Lemma 2.1]{law2014existence}   yields that
	\begin{align}\label{A.4}
	    v_n \to v \quad \text{in } C([0,T];\mathbb B_w)\subset C([0,T];\mathbb H_w),
	\end{align} 
    Combining \eqref{A.2}-\eqref{A.4},    we have extracted a subsequence, still denoted by $(v_n)$  such that $v_n\to v$ in each corresponding topology of $\mathcal Z$, hence $v_n\to v$ in the supremum topology $\mathcal T$ on $\mathcal Z$.
	Therefore, $\mathcal K$ is $\mathcal T$-relatively compact on $\mathcal Z$.
We complete  the proof. 
\end{proof}
Furthermore, by the compactness criterion in Lemma \ref{lem4.5}, we can use the same proof as \cite[Corollary 2.6]{law2014existence} to obtain the following tightness criterion.
\begin{thm}[Tightness criterion]\label{tight}
Let $(\Omega,\mathcal{F},\mathbb{P})$ be a probability space with complete and right-continuous $\{\mathcal{F}_t\}_{t\ge 0}$. Let $(X_n)_{n\in\mathbb{N}}$ be a sequence  of continuous $\{\mathcal{F}_t\}_{t\ge 0}$-adapted  $\mathbb{U}'$-valued processes such that
\begin{itemize}
    \item[(a)] there exists a positive constant $C_1$ such that
\[
\sup_{n\in\mathbb{N}} \mathbb{E}\Bigl(\sup_{s\in[0,T]}\|X_n(s)\|_{\mathbb{H}}\Big)\le C_1,
\]
\item[(b)] there exists a positive constant $C_2$ such that
\[
\sup_{n\in\mathbb{N}} \mathbb{E}\Big(\int_{0}^{T}\|X_n(s)\|_{\mathbb{V}}^{q}\,ds\Big)\le C_2,
\]
\item[(c)] there exists a positive constant $C_3$ such that
\[
\sup_{n\in\mathbb{N}} \mathbb{E}\Big(\int_{0}^{T}\|X_n(s)\|_E^{r}\,ds\Big)\le C_3,
\]
\item[(d)] $(X_n)_{n\in\mathbb{N}}$ satisfies the Aldous condition in $\mathbb{U}'$, i.e.for any $\varepsilon>0$, $\eta>0$, there exists $\delta>0$ such that for every sequence $(\tau_n)_{n\in\mathbb{N}}$ of $\{\mathcal{F}_t\}_{t\ge 0}$-stopping times with $\tau_n\le T$, it holds that
\[\sup_{n\in\mathbb{N}}\sup_{0<\theta<\delta}\mathbb{P}\big\{\|X_n(\tau_n+\theta)-X_n(\tau_n)\|_{\mathbb{U}'}\ge \eta\big\}\le\varepsilon.\]
\end{itemize}
Let $\widetilde{\mathbb{P}}_{n}$ be the law of $X_n$ on $\mathcal{Z}$.
Then for every $\varepsilon>0$ there exists a compact subset $K_{\varepsilon}$ of
$\mathcal{Z}$ such that
\[
\widetilde{\mathbb{P}}_{n}(K_{\varepsilon})\ge 1-\varepsilon.
\]
\end{thm}

    \nocite{*}
	\bibliography{main} 

@book{liu2015stochastic,
  title={Stochastic partial differential equations: an introduction},
  author={Liu, W. and R{\"o}ckner, M.},
  year={2015},
  publisher={Springer}
}

@book{grisvard2011elliptic,
  title={Elliptic problems in nonsmooth domains},
  author={Grisvard, P.},
  year={2011},
  publisher={SIAM}
}

@book{gilbarg1977elliptic,
  title={Elliptic partial differential equations of second order},
  author={Gilbarg, D. and Trudinger, N.},
year={2001},
  publisher={Springer}
}

@book{brezis2011functional,
  title={Functional analysis, Sobolev spaces and partial differential equations},
  author={Brezis, H.},
  year={2011},
  publisher={Springer}
}

@book{gawarecki2010stochastic,
  title={Stochastic differential equations in infinite dimensions: with applications to stochastic partial differential equations},
  author={Gawarecki, L. and Mandrekar, V.},
  year={2010},
  publisher={Springer Science \& Business Media}
}

@article{keller1970initiation,
  title={Initiation of slime mold aggregation viewed as an instability},
  author={Keller, E. and Segel, L.},
  journal={J.Theoret.Biol.},
  volume={26},
  number={3},
  pages={399--415},
  year={1970},
  publisher={Elsevier}
}

@article{keller1971model,
  title={Model for chemotaxis},
  author={Keller, E. and Segel, L.},
  journal={J.Theoret.Biol.},
  volume={30},
  number={2},
  pages={225--234},
  year={1971},
  publisher={Elsevier}
}

@article{Tello07062007,
author = {Tello, J.I. and  Winkler, M.},
title = {A chemotaxis system with logistic ource},
journal = {  Commun. Partial Differ. Equ.},
volume = {32},
number = {6},
pages = {849--877},
year = {2007},
publisher = {Taylor \& Francis},
 }

@article{negreanu2013competitive,
  title={On a competitive system under chemotactic effects with non-local terms},
  author={  Negreanu, M. and Tello, J. I.},
  journal={Nonlinearity},
  volume={26},
  number={4},
  pages={1083},
  year={2013},
  publisher={IOP Publishing}
}

@article{bian2018nonlocal,
  title={Nonlocal nonlinear reaction preventing blow-up in supercritical case of chemotaxis system},
  author={Bian, S. and Chen, L. and Latos, E. A},
  journal={Nonlinear Anal.},
  volume={176},
  pages={178--191},
  year={2018},
  publisher={Elsevier}
}

@article{stevens2000derivation,
  title={The derivation of chemotaxis equations as limit dynamics of moderately interacting stochastic many-particle systems},
  author={Stevens, A.},
  journal={SIAM J. APPL. MATH.},
  volume={61},
  number={1},
  pages={183--212},
  year={2000},
  publisher={SIAM}
}

@article{stevens2000stochastic,
  title={A stochastic cellular automaton modeling gliding and aggregation of myxobacteria},
  author={Stevens, A.},
  journal={SIAM J. APPL. MATH.},
  volume={61},
  number={1},
  pages={172--182},
  year={2000},
  publisher={SIAM}
}

@article{huang2021microscopic,
  title={The microscopic derivation and well-posedness of the stochastic Keller-Segel equation},
  author={Huang, H. and Qiu, J.},
  journal={J. Nonlinear Sci.},
  volume={31},
  number={1},
  pages={6},
  year={2021},
  publisher={Springer}
}

@article{flandoli2021delayed,
  title={Delayed blow-up by transport noise},
  author={Flandoli, F. and Galeati, L. and Luo, D.},
  journal={Commun. Partial. Differ. Equ.},
  volume={46},
  number={9},
  pages={1757--1788},
  year={2021},
  publisher={Taylor \& Francis}
}

@article{bresch2019mean,
  title={On mean-field limits and quantitative estimates with a large class of singular kernels: Application to the Patlak-Keller-Segel model},
  author={Bresch, D. and Jabin, P-E. and Wang, Z.},
  journal={C. R. Math.},
  volume={357},
  number={9},
  pages={708--720},
  year={2019},
  publisher={Elsevier}
}

@article{misiats2022global,
  title={On global existence and blowup of solutions of stochastic Keller-Segel type equation},
  author={Misiats, O. and Stanzhytskyi, O. and Topaloglu, I.},
  journal={Nonlinear Differ. Equ. Appl.},
  volume={29},
  number={1},
  pages={3},
  year={2022},
  publisher={Springer}
}

@article{hausenblas2022one,
  title={The one-dimensional stochastic Keller-Segel model with time-homogeneous spatial Wiener processes},
  author={Hausenblas, E. and Mukherjee, D. and Tran, T.},
  journal={J. Differ. Equ},
  volume={310},
  pages={506--554},
  year={2022},
  publisher={Elsevier}
}

@article{tang2024strong,
  title={Strong solutions to a nonlinear stochastic aggregation-diffusion equation},
  author={Tang, H. and Wang, Z-A.},
  journal={Commun. Contemp. Math.},
  volume={26},
  number={02},
  pages={2250073},
  year={2024},
  publisher={World Scientific}
}

@article{chen2025well,
  title={Well-posedness of stochastic chemotaxis system},
  author={Chen, Y. and Zhai, J. and Zhang, T.},
  journal={J. Differ. Equ.},
  volume={442},
  pages={113531},
  year={2025},
  publisher={Elsevier}
}

@article{dhariwal2019global,
  title={Global martingale solutions for a stochastic population cross-diffusion system},
  author={Dhariwal, G. and J{\"u}ngel, A. and Zamponi, N.},
  journal={Stoch. Process. and  Appl.},
  volume={129},
  number={10},
  pages={3792--3820},
  year={2019},
  publisher={Elsevier}
}

@article{law2014existence,
  title={The existence of martingale solutions to the stochastic Boussinesq equations},
  author={Brze{\'z}niaka, Z. and Motylb, E.},
  journal={Global. Stoch. Anal.},
  volume={1},
  number={2},
  pages={175--216},
  year={2014}
}

@article{brzezniak2013existence,
  title={Existence of a martingale solution of the stochastic Navier-Stokes equations in unbounded {$2D$} and {$3D$} domains},
  author={Brze{\'z}niak, Z. and Motyl, E.},
  journal={J. Differ. Equ.},
  volume={254},
  number={4},
  pages={1627--1685},
  year={2013},
  publisher={Elsevier}
}

@article{brzezniak2017invariant,
  title={Invariant measure for the stochastic Navier-Stokes equations in unbounded {$2D$ } domains},
  author={Brze{\'z}niak, Z. and Motyl, E. and Ondrejat, M.},
journal={Ann.Probab.},
  volume={45},
  number={5},
  pages={3145–3201},
  year={2017}
}

@article{brzeźniak2019fractionally,
  title={Fractionally Dissipative Stochastic Quasi-Geostrophic Type Equations on {R\^{d}}},
  author={Brze{\'z}niak, Z. and Motyl, E.},
  journal={SIAM J. MATH. ANAL.},
  volume={51},
  number={3},
  pages={2306--2358},
  year={2019},
  publisher={SIAM}
}

@article{braukhoff2024global,
  title={Global martingale solutions for stochastic Shigesada-Kawasaki-Teramoto population models},
  author={Braukhoff, M. and Huber, F. and J{\"u}ngel, A.},
  journal={Stoch PDE: Anal Comp},
  volume={12},
  number={1},
  pages={525--575},
  year={2024},
  publisher={Springer}
}

@article{tang2018pathwise,
  title={On the Pathwise Solutions to the Camassa-Holm Equation with Multiplicative Noise},
  author={Tang, H.},
  journal={SIAM J. APPL. MATH.},
  volume={50},
  number={1},
  pages={1322--1366},
  year={2018},
  publisher={SIAM}
}

@article{grieser2002uniform,
  title={Uniform bounds for eigenfunctions of the Laplacian on manifolds with boundary},
  author={Grieser, D.},
  journal={Commun. Partial Differ. Equ},
  volume={27},
  number={7-8},
  pages={1283--1299},
  year={2002},
  publisher={Taylor \& Francis}
}

@article{kroger1992upper,
  title={Upper bounds for the Neumann eigenvalues on a bounded domain in Euclidean space},
  author={Kr{\"o}ger, P.},
  journal={J. Funct. Anal.},
  volume={106},
  number={2},
  pages={353--357},
  year={1992},
  publisher={Elsevier}
}

@article{bian2016nonlocal,
  title={A nonlocal reaction diffusion equation and its relation with Fujita exponent},
  author={Bian, S. and Chen, L.},
  journal={J. Math. Anal. Appl.},
  volume={444},
  number={2},
  pages={1479--1489},
  year={2016},
  publisher={Elsevier}
}

@article{kuehn2020pathwise,
  title={Pathwise mild solutions for quasilinear stochastic partial differential equations},
  author={Kuehn, C. and Neam{\c{t}}u, A.},
  journal={J. Differ. Equ.},
  volume={269},
  number={3},
  pages={2185--2227},
  year={2020},
  publisher={Elsevier}
}

@article{krylov2010ito,
  title={{It{\^o}’s} formula for the {$L^{p}$}-norm of stochastic {$W^{1,p}$}-valued processes.},
  author={Krylov, N. V.},
  journal={Probab. Theory Relat. Fields.},
  volume={147},
pages={583--605},
  year={2010}
}

@article{fischer2018existence,
  title={Existence of positive solutions to stochastic thin-film equations},
  author={Fischer, J. and Gru{\"u}n, G.},
  journal={SIAM J. MATH. ANAL.},
  volume={50},
  number={1},
  pages={411--455},
  year={2018},
  publisher={SIAM}
}

@article{debussche2011local,
  title={Local martingale and pathwise solutions for an abstract fluids model},
  author={Debussche, A. and Glatt-Holtz, N. and Temam, R.},
  journal={Phys. D},
  volume={240},
  number={14-15},
  pages={1123--1144},
  year={2011},
  publisher={Elsevier}
}

@article{salins2022global,
  title={Global solutions for the stochastic reaction-diffusion equation with super-linear multiplicative noise and strong dissipativity},
  author={Salins, M.},
  journal={Electron. J. Probab.},
  volume={27},
  pages={1--17},
  year={2022},
  publisher={The Institute of Mathematical Statistics and the Bernoulli Society}
}

@article{zhai20202d,
  title={{$2D$} stochastic chemotaxis-Navier-Stokes system},
  author={Zhai, J. and Zhang, T.},
  journal={J.  Math. Pure. Appl.},
  volume={138},
  pages={307--355},
  year={2020},
  publisher={Elsevier}
}

@article{chen2025mean,
  title={Mean-Field Control for Diffusion Aggregation Equation with Coulomb Interaction},
  author={Chen, L. and Wang, Y. and Wang, Z.},
  journal={SIAM. J. CONTROL. OPTIM.},
  volume={63},
  number={4},
  pages={3061--3090},
  year={2025},
  publisher={SIAM}
}

@article{chen2026fluctuations,
  title={Fluctuations around the mean-field limit for attractive Riesz potentials in the moderate regime},
  author={Chen, L. and Holzinger, A. and J{\"u}ngel, A.},
  journal={ARCH. RATION. MECH. AN.},
  volume={250},
  number={1},
  pages={5},
  year={2026},
  publisher={Springer}
}

@article{nikolaev2025quantitative,
  title={Quantitative relative entropy estimates for interacting particle systems with common noise},
  author={Nikolaev, P.},
  journal={SIAM. J. MATH. ANAL.},
  volume={57},
  number={3},
  pages={3071--3109},
  year={2025},
  publisher={SIAM}
}

@article{chen2025hegselmann,
  title={Hegselmann-Krause model with environmental noise},
  author={Chen, L. and Nikolaev, P. and Pr{\"o}mel, D.},
  journal={T. Am. MATH. SOC.},
  volume={378},
  number={01},
  pages={527--567},
  year={2025}
}

@article{bresch2023mean,
  title={Mean field limit and quantitative estimates with singular attractive kernels},
  author={Bresch, D. and Jabin, P. and Wang, Z.},
  journal={Duke. Math. J.},
  volume={172},
  number={13},
  pages={2591--2641},
  year={2023},
  publisher={Duke University Press}
}

@incollection{carrillo2014derivation,
  title={The derivation of swarming models: mean-field limit and Wasserstein distances},
  author={Carrillo, J.A. and Choi, Y-P. and Hauray, M.},
  booktitle={Collective Dyn. Bacteria Crowds.},
  pages={1--46},
year={2014},
  publisher={Springer}
}

@incollection{jabin2017mean,
  title={Mean field limit for stochastic particle systems},
  author={Jabin, P-E. and Wang, Z.},
  booktitle={Active Particles, Vol. 1.},
  pages={379--402},
  year={2017},
  publisher={Springer}
}
	\bibliographystyle{plain}  

\end{document}